\documentclass[11pt,english]{article}
\usepackage{color}
\usepackage[T1]{fontenc}
\usepackage[latin9]{inputenc}
\usepackage[a4paper]{geometry}
\usepackage{textcomp}
\usepackage{amsthm}
\usepackage{amsmath}
\usepackage{amssymb}
\usepackage{esint}
\usepackage{bbm}
\usepackage{hyperref}
\usepackage{b abel}
\usepackage{amscd}
\usepackage{amsfonts}
\usepackage{amsthm}
\usepackage{mathrsfs}
\usepackage{dsfont}
\usepackage{mathtools}
\usepackage{enumerate}
\usepackage{bm}
\usepackage{bbm}
\usepackage{authblk}
\usepackage[inline]{enumitem}

\usepackage[english=american]{csquotes}

\usepackage{cite}

\numberwithin{equation}{section}

\DeclareMathOperator{\Tr}{Tr}
\DeclareMathOperator{\OpW}{Op_\hbar^w}
\DeclareMathOperator{\supp}{supp}
\DeclareMathOperator{\Vol}{Vol}
\DeclareMathOperator{\dist}{dist}
\DeclareMathOperator{\re}{\mathrm{Re}}
\DeclareMathOperator{\im}{\mathrm{Im}}

\newcommand{\norm}[1]{\lVert #1 \rVert}
\newcommand{\abs}[1]{\left| #1 \right|}

\newcommand{\R}{\mathbb{R}}
\newcommand{\C}{\mathbb{C}}
\newcommand{\N}{\mathbb{N}}
\newcommand{\Z}{\mathbb{Z}}

\newtheorem{thm}{Theorem}[section]
\newtheorem{lemma}[thm]{Lemma}
\newtheorem{proposition}[thm]{Proposition}

\theoremstyle{definition}
\newtheorem{definition}[thm] {Definition}
\newtheorem{assumption}[thm]{Assumption}

\theoremstyle{remark}
\newtheorem{remark}[thm]{Remark}

\begin{document}

\title{Sharp semiclassical spectral asymptotics for local magnetic Schr\"odinger operators on $\R^d$ without full regularity.}

\author{S{\o}ren Mikkelsen}
   \affil{\small{Department of Mathematics and Statistics, University of Helsinki \\ Helsinki, Finland \\ email:~\texttt{soren.mikkelsen@helsinki.fi} }}

\maketitle

\begin{abstract}
We consider operators acting in $L^2(\R^d)$ with $d\geq3$ that locally behave as a magnetic Schr\"odinger operator. For the magnetic Schr\"odinger operators we suppose the magnetic potentials are smooth and the electric potential is five times differentiable and the fifth derivatives are H\"older continuous. Under these assumptions, we establish sharp spectral asymptotics for localised counting functions and Riesz means. 
\end{abstract}


\section{Introduction}
We will here consider sharp semiclassical spectral asymptotics for operators $\mathcal{H}_{\hbar,\mu}$ that locally are given by a magnetic Schr\"odinger operators acting in $L^2(\R^d)$ for $d\geq3$. What we precisely mean by ``locally given by'' will be clarified below.  That is we consider operators that locally are of the form
\begin{equation}\label{def_op_intro}
	H_{\hbar,\mu} = (-i\hbar\nabla - \mu a)^2  + V, 
\end{equation}
where $\hbar\in(0,1]$ is the semiclassical parameter, $\mu\geq0$ is the intensity of the magnetic field, $a$ is the magnetic vector potential and  $V$ is the electric potential. Our exact assumptions on the potentials and intensity $\mu$ will be stated below. We will here for $\gamma\in[0,1]$ be interested in the asymptotics as $\hbar$ goes to zero of the following traces
\begin{equation}\label{traces_to_consider}
	\Tr[\varphi g_\gamma(\mathcal{H}_{\hbar,\mu})],
\end{equation}
for $\gamma\in[0,1]$, where $\varphi\in C_0^\infty(\R^d)$. The function $g_\gamma$ is given by
\begin{equation}
	g_\gamma(t) = \begin{cases}
	\boldsymbol{1}_{(-\infty,0]}(t) &\gamma=0
	\\
	(t)_{-}^\gamma &\gamma\in(0,1],
	\end{cases}
\end{equation}
where  we have used the notation $(x)_{-} = \max(0,-x)$ and $\boldsymbol{1}_{(-\infty,0]}$ is the characteristic function for the set $(-\infty,0]$. To ensure that the leading order term in the asymptotics is independent of the magnetic field we will assume that $\hbar\mu\leq C$, where $C$ is some positive constant. These localised traces can be used to understand the global quantity 
\begin{equation}\label{traces_to_consider_global}
	\Tr[g_\gamma(H_{\hbar,\mu})].
\end{equation}
For the global quantity $\Tr[g_\gamma(H_{\hbar,1})]$, where the magnetic potential is independent of the semiclassical parameter, it was established by Helffer and Robert in \cite{MR724029} for the case $\gamma=0$ and \cite{MR1079775} for the case $\gamma\in(0,1]$ that
\begin{equation*}
	  \Big|\Tr[g_\gamma(H_{\hbar,1})] - \frac{1}{(2\pi\hbar)^d} \int_{\R^{2d}}g_\gamma(p^2+V(x)) \,dx dp \Big| \leq C \hbar^{1+\gamma-d}.
\end{equation*}
for $\hbar$ sufficiently small. They established these results under the assumption that both the electric potential and the magnetic vector potential are smooth functions plus some additional regularity assumptions. Their result is sharp in the sense that the error is the best one can obtain without more assumptions. This can be seen by explicitly diagonalising the Harmonic oscillator. For the case, where the operator $\mathcal{H}_{\hbar,\mu}$ is locally given by a magnetic Schr\"odinger operator with smooth electric and the magnetic vector potentials it was established by Sobolev in \cite{MR1343781} that
\begin{equation*}
	  \Big|\Tr[\varphi g_\gamma(\mathcal{H}_{\hbar,\mu})] - \frac{1}{(2\pi\hbar)^d} \int_{\R^{2d}}g_\gamma(p^2+V(x)) \varphi(x) \,dx dp \Big| \leq C \langle \mu \rangle^{1+\gamma} \hbar^{1+\gamma-d}
\end{equation*}
for $\hbar$ sufficiently small, where the strength of the magnetic field $\mu$ satisfies that $\hbar\mu\leq C$ for some constant $C>0$. What we precisely mean by ``locally given by'' will again be clarified below. The aim of this paper is to obtain these sharp bounds without assuming the electric potential is smooth. This is not the first work to consider sharp spectral asymptotic without full regularity and we will comment on the literature in Section~\ref{SEC:previous_work}.

Spectral asymptotics of the quantities $\Tr[\varphi g_\gamma(\mathcal{H}_{\hbar,\mu})] $ and $\Tr[g_\gamma(H_{\hbar,\mu})] $ are not just of mathematical interest, they are also of physical interest.
Especially the case $\gamma=1$ has physical motivation both with and without a magnetic vector potential. For details see e.g. \cite{MR1272387,MR1266071,PhysRevLett.69.749,MR1181242,MR428944,MR2013804}. The case $\gamma=0$ is also of interest. Recently in \cite{MR4182014} sharp estimates for the trace norm of commutators between spectral projections and position and momentum operators were obtained using asymptotic for \eqref{traces_to_consider} for $\gamma=0$. This type of bound first appeared as an assumption in \cite{MR3248060}, where the mean-field evolution of fermionic systems was studied. The assumption has also appeared in  \cite{MR3570479,MR3202863,MR3381147,MR3461406,MR4009687,MR4602009}.  
The asymptotics used in \cite{MR4182014} were obtained in \cite{MR1343781}.

\subsection{The main results}

Before we state our main result we will specify our assumptions on the operator $\mathcal{H}_{\hbar,\mu}$ and what we mean by ``locally given by a magnetic Schr\"odinger operator''. That we only locally assume $\mathcal{H}_{\hbar,\mu}$ is acting as a magnetic Schr\"odinger operator is due to the presence of the cut-off function. This type of assumption first appeared in  \cite{MR1343781} to the knowledge of the author. Our exact assumptions are given below.
\begin{assumption}\label{Assumption:local_potential_1}
 Let $\mathcal{H}_{\hbar,\mu}$ be an operator acting in $L^2(\R^d)$, where $\hbar>0$ and $\mu\geq0$. Moreover, let $\gamma\in[0,1]$. Suppose that
\begin{enumerate}[label={$\roman*)$}]
  \item\label{G.L.1.1}  $\mathcal{H}_{\hbar,\mu}$ is self-adjoint and lower semibounded.
  \item\label{G.L.1.2}  Suppose there exists an open set $\Omega\subset\R^d$ and real valued functions $V\in C^{5,\kappa}_0(\R^d)$ with $\kappa>\gamma$, $a_j\in C_0^\infty(\R^d)$ for $j\in\{1,\dots,d\}$ such that $C_0^\infty(\Omega)\subset\mathcal{D}(\mathcal{H})$ and
\begin{equation*}
	\mathcal{H}_{\hbar,\mu} \varphi = H_{\hbar,\mu}\varphi \quad\text{for all $\varphi\in C_0^\infty(\Omega)$},
\end{equation*}
where $H_{\hbar,\mu}= (-i\hbar\nabla - \mu a)^2  + V$.
  \end{enumerate}
\end{assumption}
In the assumption, we have used the notation $C_0^{5,\kappa}(\R^d)$. This is the space of compactly supported functions that are five times differentiable and the fifth derivatives are uniformly H\"older continuous with parameter $\kappa$. That is for $f\in C_0^{5,\kappa}(\R^d)$ there exists a constant $C>0$ such that for all $x,y\in\R^d$ it holds that
  \begin{equation}\label{EQ:Holder_def}
  	\begin{aligned}
	 |\partial_x^{\alpha} f(x) - \partial_x^{\alpha} f(y) |  \leq C |x-y|^{\kappa} \quad\text{for all $\alpha \in\N_0^d$  with  $\abs{\alpha}=5$.}
    \end{aligned}
  \end{equation}
Note that we here and in the following are using the convention that $\N$ does not contain $0$ and we will use the notation $\N_0 = \N\cup\{0\}$. Moreover for the cases where $\kappa>1$ we use the convention that  
\begin{equation}
	C_0^{5,\kappa}(\R^d) \coloneqq C_0^{5 + \lfloor \kappa \rfloor, \kappa-\lfloor \kappa \rfloor}(\R^d),
\end{equation}
where $C_0^{k,\kappa}(\R^d)$ is the space of compactly supported functions that are $k$ times differentiable and the $k$'ed derivatives are uniformly H\"older continuous with parameter $\kappa$. We will in what follows denote the constant $C$ in \eqref{EQ:Holder_def} for the H\"older constant for $f$. The assumptions we make on the operator $\mathcal{H}_{\hbar,\mu}$ are very similar to the assumptions made in \cite{MR1343781}.
The difference is that we do not require $V$ to be smooth. Instead, we assume it has five derivatives and that the fifth derivative is uniformly H\"older continuous.  
With this assumption in place, we can state our main result.
\begin{thm}\label{THM:Local_five_derivative}
Let $\mathcal{H}_{\hbar,\mu}$ be an operator acting in $L^2(\R^d)$ and let $\gamma\in[0,1]$. If $\gamma=0$ we assume $d\geq3$ and if $\gamma\in(0,1]$ we assume $d\geq4$. Suppose that $\mathcal{H}_{\hbar,\mu}$ satisfies Assumption~\ref{Assumption:local_potential_1} with the set $\Omega$ and the functions $V$ and $a_j$ for $j\in\{1,\dots,d\}$. 
Then for any $\varphi\in C_0^\infty(\Omega)$ it holds that
\begin{equation*}
	  \Big|\Tr[\varphi g_\gamma(\mathcal{H}_{\hbar,\mu})] - \frac{1}{(2\pi\hbar)^d} \int_{\R^{2d}}g_\gamma(p^2+V(x)) \varphi(x) \,dx dp \Big| \leq C \langle \mu \rangle^{1+\gamma} \hbar^{1+\gamma-d}
\end{equation*}
for all $\hbar\in(0,\hbar_0]$ and $\mu\leq C\hbar^{-1}$, where $\hbar_0$ is sufficiently small. With the notation $\langle \mu \rangle=(1+\mu^2)^{\frac12}$. The constant $C$ depends on the dimension $d$, the numbers $\gamma$, $\norm{\varphi}_{L^\infty(\R^d)}$, $\norm{\partial^\alpha_x  \varphi}_{L^\infty(\R^d)}$ and  $\norm{\partial^\alpha a_j}_{L^\infty(\R^d)}$ for all $\alpha\in\N_0^d$ with $|\alpha|\geq1$ and $j\in\{1\dots,d\}$, $\norm{\partial_x^\alpha V}_{L^\infty(\R^d)}$ for all $\alpha\in\N_0^d$ such that $|\alpha|\leq 5$ and the H\"older constant for $V$.
\end{thm}
\begin{remark}
We remark that the error term is independent of $\norm{ a_j}_{L^\infty(\R^d)}$ for all $j\in\{1\dots,d\}$. This is also the case for the results in  \cite{MR1343781}. As remarked in \cite{MR1343781} it is not surprising as the magnitude of $a_j$ can easily be changed by a Gauge transform.

In this result we allow the strength of the magnetic field to depend on the semiclassical parameter. As expected we observe that as the strength of the magnetic field increases the ``worse'' the error becomes. When $\mu=C\hbar^{-1}$ we can also see that the error is of the same order as our leading order term. Hence in this regime, the leading order term has to be corrected to obtain an error of lower order. The leading order term for the energy (the case $\gamma=1$) taking into account the magnetic field was first rigorously derived by Lieb, Solovej, and Yngvarson in \cite{MR1266071} see also \cite{MR4529872}. For two terms asymptotics with the corrected leading order term see \cite{MR1425276} and \cite[Vol III and IV]{ivrii2019microlocal1}     
 
It would have been desirable not to assume the magnetic vector potential to be smooth. However, with the techniques we use here, this is not possible. See Remark~\ref{RE:approximations} part~\ref{RE:approximations_2} for further details.  
\end{remark}
The assumptions on the dimension are needed to ensure that certain integrals will be convergent independent of $\hbar$. The next result will consider the cases of dimensions $2$ and $3$. However, these will not be sharp except in the case $d=3$ and $\gamma=0$. 
The following result is proven using almost analogous arguments to those used to prove Theorem~\ref{THM:Local_five_derivative}. We will not give separate full proofs of this result. Instead, we will in Remark~\ref{RE:proof_non_sharp_thm} describe how to alter the proof of the main result to obtain these.
\begin{thm}\label{THM:Main_non_sharp}
Let $\mathcal{H}_{\hbar,\mu}$ be an operator acting in $L^2(\R^d)$, with $d=2$ or $d=3$ and let $\gamma\in[0,1]$. Suppose that $\mathcal{H}_{\hbar,\mu}$ satisfies Assumption~\ref{Assumption:local_potential_1} with the set $\Omega$ and the functions $V$ and $a_j$ for $j\in\{1,\dots,d\}$. 
Then for any $\varphi\in C_0^\infty(\Omega)$ it holds that
\begin{equation*}
	  \Big|\Tr[\varphi g_\gamma(\mathcal{H}_{\hbar,\mu})] - \frac{1}{(2\pi\hbar)^d} \int_{\R^{2d}}g_\gamma(p^2+V(x)) \varphi(x) \,dx dp \Big| \leq C \langle \mu \rangle^{\frac{d+2\gamma}{3}} \hbar^{\frac{2}{3}(\gamma-d)}
\end{equation*}
for all $\hbar\in(0,\hbar_0]$ and $\mu\leq C\hbar^{-1}$, where $\hbar_0$ is sufficiently small. With the notation $\langle \mu \rangle=(1+\mu^2)^{\frac12}$. The constant $C$ depends on the dimension $d$, the numbers $\gamma$, $\norm{\varphi}_{L^\infty(\R^d)}$, $\norm{\partial^\alpha_x  \varphi}_{L^\infty(\R^d)}$ and  $\norm{\partial^\alpha a_j}_{L^\infty(\R^d)}$ for all $\alpha\in\N_0^d$ with $|\alpha|\geq1$ and $j\in\{1\dots,d\}$, $\norm{\partial_x^\alpha V}_{L^\infty(\R^d)}$ for all $\alpha\in\N_0^d$ such that $|\alpha|\leq 5$ and the H\"older constant for $V$.
\end{thm}
One thing to note is that in the case $d=3$ and $\gamma=0$ this result still gives us a sharp estimate. For the remaining cases, this result does not yield sharp bounds in terms of the semiclassical parameter. For the case $d=3$ and $\gamma=1$ we get the error that is $\langle \mu \rangle^{\frac{5}{3}} \hbar^{-\frac{4}{3}}$ in terms of the semiclassical parameter and the strength of the magnetic field. For the sharp bound, we would expect an error of the form $\langle \mu \rangle^{2} \hbar^{-1}$.  

\subsection{Previous results and outline of the paper}\label{SEC:previous_work}

The first sharp results on spectral asymptotics were, as mentioned, obtained by Helffer and Robert in \cite{MR724029,MR1079775}. They considered general $\hbar$-pseudo-differential operators that include magnetic Schr\"odinger operators with the intensity of the magnetic field that is independent of the semiclassical parameter. Sharp spectral asymptotics for operators satisfying Assumption~\ref{Assumption:local_potential_1} with $V\in C_0^\infty(\R^d)$ was established by Sobolev in \cite{MR1343781}. In \cite{MR1412359} sharp asymptotics were also obtained however, the electric potential was allowed to be singular at the origin but otherwise smooth. 

We will not give a full review of the literature on sharp spectral asymptotics without full regularity but focus on the literature also considering magnetic Schr\"odinger operators. For a more historical review of the literature we refer the reader to the introduction of \cite{MR4689394}. 

In \cite{MR1974450} Bronstein and Ivrii consider differential operators on $L^2(M)$, where $M$ is a compact manifold without a boundary. They assume the coefficients are differentiable and the first derivatives are continuous with continuity modulus $\mathcal{O}( \log |x-y|^{-1})$. Under a non-critical condition\footnote{In the paper this condition is not denoted non-critical but microhyperbolic.} they establish sharp spectral asymptotics for the counting function. that is the case $\gamma=0$. It is mentioned in a remark that their results of the paper extend to $\gamma\in(0,1]$ if the techniques used in the paper are combined with techniques from \cite{MR1631419}. Their result includes magnetic Schr\"odinger operators with the intensity of the magnetic field that is independent of the semiclassical parameter. The non-critical condition for a magnetic Schr\"odinger operator is equivalent to assuming that
\begin{equation*}
	|V(x)| \geq c >0, \qquad\text{for all $x\in M$}. 
\end{equation*}
We cannot have that this assumption is verified and have only pure point spectrum on the negative half axis when the underlying space is non-compact. So these results do not immediately generalise to the non-compact setting. In \cite{MR1974451} Ivrii considers the same general setting as in \cite{MR1974450} and only the case $\gamma=0$ is considered. The manifold is still assumed to be compact but is allowed to have a boundary. Moreover, the non-critical assumption is removed. The techniques used to remove the non-critical condition are based on multi-scale analysis and we also will use these techniques here. In \cite{MR2335576} Zielinski obtained sharp spectral asymptotics for the counting function for differential operators with non-smooth coefficients on $L^2(\R^d)$ with $d\geq3$ without a non-critical condition but with a geometric condition on the semiclassical principal symbol. The coefficients are in this work assumed to be bounded and twice differentiable with the second derivatives being H\"older continuous with some positive parameter. The geometric condition on the semiclassical principal symbol $a_0$ is that for some $\varepsilon>0$ there exists a constant $C>0$ such that
\begin{equation*}
	\sup_{E\in [-\hbar^{1-\varepsilon},\hbar^{1-\varepsilon}]} \Vol \big\{ (x,p)\in\R^{2d} \, \big| \, |a_0(x,p) - E| \leq \hbar \big\} \leq C\hbar.
\end{equation*} 
Again these results include magnetic Schr\"odinger operators with intensity of the magnetic field that is independent of the semiclassical parameter.
In \cite{MR2179891} Ivrii considers magnetic Schr\"odinger operators. In this work, non-smooth electric and magnetic potentials are considered, and sharp spectral asymptotics are obtained.  In this paper, the strength of the magnetic field is allowed to depend on the semiclassical parameter.
These results are also given in \cite[Vol IV]{ivrii2019microlocal1}. In some cases, the results presented in  \cite{MR2179891} and \cite[Vol IV]{ivrii2019microlocal1} require less smoothness than here. The results in \cite{MR1974450,MR1974451} can also be found in \cite[Vol I]{ivrii2019microlocal1}.    

Most of the mentioned results require less smoothness than the results presented here. The main difference is that our starting point is a trace that is localised that is we consider $\Tr[\varphi g_\gamma(\mathcal{H}_{\hbar,\mu})]$ and not $\Tr[g_\gamma(\mathcal{H}_{\hbar,\mu})]$. This localisation gives rise to some difficulties. Usually, when you want to prove sharp spectral asymptotics for a differential operator $A(\hbar)$  with non-smooth coefficients you compare quadratic forms. See e.g. \cite{MR1974450,MR2179891,MR1631419,MR1974451,MR2105486,MR4689394,MR2335576}. This is due to the observation that if you have an operator $A(\hbar)$ and two approximating or framing operators $A^{\pm}(\hbar)$ such that
  \begin{equation*}
    A^{-}(\hbar) \leq A(\hbar)  \leq A^{+}(\hbar)
  \end{equation*}
 in the sense of quadratic forms. Then by the min-max-theorem, we obtain the relation
  \begin{equation}\label{EQ:com_quad_tr}
    \Tr[g_\gamma(A^{+}(\hbar))] \leq \Tr[g_\gamma(A(\hbar)) ] \leq  \Tr[g_\gamma(A^{-}(\hbar))].
  \end{equation} 
The aim is then to choose the approximating operators such that sharp asymptotics can be obtained for these and then use \eqref{EQ:com_quad_tr} to deduce it for the original operator $A(\hbar)$. In the situation we consider, we have a localisation in the trace. This implies that we cannot get a relation like \eqref{EQ:com_quad_tr} from the min-max theorem. We will instead estimate the difference directly and prove that the traces of our original problem are sufficiently close to the trace when we have inserted the approximation.  
In Section~\ref{sec:Pre} we specify the notation we use and describe the operators we will be working with. Moreover, we recall some definitions and results that we will need later. At the end of the section, we describe how we approximate the non-smooth potential by a smooth potential. 

In Section~\ref{sec:Rough_pseudo_diff_op} we recall some results and definitions on rough $\hbar$-pseudo-differential operators. We also prove some specific results for rough Schr\"odinger operators. 

In Section~\ref{sec:Aux_est} we establish several estimates for operator satisfying Assumption~\ref{Assumption:local_potential_1}. The ideas and techniques used here are inspired by the ideas and techniques used in  \cite{MR1343781}. Some of the results will also be taken directly from \cite{MR1343781}. These auxiliary results are needed to prove a version of the main theorem under an additional non-critical condition. This version is proven in Section~\ref{sec:model_prob}. Finally in Section~\ref{sec:proof_main} we give the proof of the main theorem in two steps. First in the case where $\mu\leq\mu_0<1$ and then the general case. 
\subsection*{Acknowledgement}
The author is grateful to the Leverhulme Trust for their support via Research Project Grant 2020-037. The author is also grateful to the anonymous referees for carefully reading the manuscript and providing helpful remarks and comments that have helped improve the manuscript.
\section{Preliminaries}\label{sec:Pre}
We start by specifying some notation. For an open set $\Omega\subseteq\R^d$ we will in the following by $\mathcal{B}^\infty(\Omega)$ denote the space
\begin{equation*}
 	\mathcal{B}^\infty(\Omega) \coloneqq \big\{ \psi\in C^\infty(\Omega) \, \big| \, \norm{\partial^\alpha \psi}_{L^\infty(\Omega) }<\infty \, \forall \alpha\in\N_0^d  \big\}.
\end{equation*}
We will for an operator $A$ acting in a Hilbert space $\mathscr{H}$ denote the operator norm by $\norm{A}_{\mathrm{op}}$ and the trace norm by $\norm{A}_1$.

Next, we describe the operators we will be working with. If we have $a_j\in L^2_{loc}(\R^d)$ for all $j\in\{1,\dots,d\}$ then we can consider the following form
\begin{equation}
	\mathfrak{h}_0[f,g] = \sum_{j=1}^d \int_{\R^d} (-i\hbar \partial_{x_j} -\mu a_j(x))f(x) \overline{(-i\hbar \partial_{x_j} -\mu a_j(x))g(x)} \,dx \quad f,g\in \mathcal{D}[\mathfrak{h}_0]
\end{equation}
for $\mu\geq0$ and $\hbar>0$, where $\mathcal{D}[\mathfrak{h}_0]$ is the domain for the form. Note that $C_0^\infty(\R^d)\subset \mathcal{D}[\mathfrak{h}_0]$. Moreover, this form is closable and lower semibounded (by zero) see \cite{MR526289} for details. Hence there exists a positive self-adjoint operator associated with the form (the Friederichs extension). For details see e.g. \cite{MR0493420} or \cite{MR526289}. We will by $\mathcal{Q}_j$ denote the square root of this operator.  When we also have a potential $V\in L^\infty(\R^d)$ we can define the operator $H_{\hbar,\mu}$ as the Friederichs extension of the quadratic form 
\begin{equation}
	\mathfrak{h}[f,g] =  \int_{\R^d} \sum_{j=1}^d(-i\hbar \partial_{x_j} -\mu a_j(x))g(x) \overline{(-i\hbar \partial_{x_j} -\mu a_j(x))g(x)} + V(x)f(x)\overline{g(x)} \,dx \quad f,g\in \mathcal{D}[\mathfrak{h}]
\end{equation}
for $\mu\geq0$ and $\hbar>0$, where $\mathcal{D}[\mathfrak{h}]$ is the domain for the form. This construction gives us that $H_{\hbar,\mu}$ is self-adjoint and lower semibounded. Again for details see e.g. \cite{MR0493420} or \cite{MR526289}.

When working with the Fourier transform we will use the following semiclassical
version for $\hbar>0$
\begin{equation*}
  \mathcal{F}_\hbar[ \varphi ](p) \coloneqq \int_{\R^d} e^{-i\hbar^{-1} \langle x,p \rangle}\varphi(x) \, dx,
\end{equation*}
and with inverse given by
\begin{equation*}
  \mathcal{F}_\hbar^{-1}[\psi] (x) \coloneqq \frac{1}{(2\pi\hbar)^d} \int_{\R^d} e^{i\hbar^{-1} \langle x,p \rangle}\psi(p) \, dp,
\end{equation*}
where $\varphi$ and $\psi$ are elements of $\mathcal{S}(\R^d)$. Here $\mathcal{S}(\R^d)$ denotes the Schwartz space.

We will for some of the results see that they are true for a larger class of functions containing $g_\gamma$. These classes were first defined in \cite{MR1343781,MR1272980} and we recall the definition here. 
\begin{definition}
A function $g\in C^\infty(\R\setminus\{0\})$ is said to belong to the class $C^{\infty,\gamma}(\R)$, $\gamma\in[0,1]$, if $g\in C(\R)$ for $\gamma>0$, for some constants $C >0$ and $r>0$ it holds that
\begin{equation*}
	\begin{aligned}
	g(t) &= 0, \qquad\text{for all $t\geq C$}
	\\
	|\partial_t^m g(t)| &\leq C_m |t|^r, \qquad\text{for all $m\in\N_0$ and $t\leq -C$}
	\\
	|\partial_t^m g(t)| &\leq
	\begin{cases} 
		C_m  & \text{if $\gamma=0,1$} \\  
		C_m|t|^{\gamma-m} &\text{if $\gamma\in(0,1)$} 
	\end{cases}, 
	\qquad\text{for all $m\in\N$ and $t\in [ -C,C]\setminus\{0\}$}.
	\end{aligned}
\end{equation*}
A function $g$ is said to belong to $C^{\infty,\gamma}_0(\R)$ if $g\in C^{\infty,\gamma}(\R)$ and $g$ has compact support.
\end{definition}

We will in our analysis need different ways for expressing functions of self-adjoint operators. One of these is the Helffer-Sj\"ostrand formula. Before we state it we will recall a definition of an almost analytic extension. 
\begin{definition}\label{Def:almost_analytic_function}[Almost analytic extension]
For $f\in C_0^\infty(\R)$ we call a function $\tilde{f} \in C_0^\infty(\C)$ an almost analytic extension if it has the properties 
\begin{equation*}
	\begin{aligned}
	|\bar{\partial} \tilde{f}(z)| &\leq C_n |\im(z)|^n, \qquad \text{for all $n\in\N_0$}
	\\
	\tilde{f}(t)&=f(t) \qquad \text{for all $t\in\R$},
	\end{aligned}
\end{equation*}
where $\bar{\partial} = \frac12 (\partial_x +i\partial_y)$.
\end{definition}
For how to construct the almost analytic extension for a given $f\in  C_0^\infty(\R)$ see e.g. \cite{MR2952218,MR1735654}. The following theorem is a simplified version of a theorem in \cite{MR1349825}.
\begin{thm}[The Helffer-Sj\"{o}strand formula]\label{THM:Helffer-Sjostrand}
  Let $H$ be a self-adjoint operator acting on a Hilbert space $\mathscr{H}$ and $f$ a function from $C_0^\infty(\R)$. Then the bounded operator $f(H)$ is given by the equation
  \begin{equation*}
    f(H) =- \frac{1}{\pi} \int_\C   \bar{\partial }\tilde{f}(z) (z-H)^{-1} \, L(dz),
  \end{equation*}
  where $L(dz)=dxdy$ is the Lebesgue measure on $\C$ and $\tilde{f}$ is an almost analytic extension of $f$.
\end{thm}

\subsection{Approximation of the potential}
In our analysis, we will need to approximate the potential with a smooth potential. How we choose this approximation is the content of the next lemma.  
\begin{lemma}\label{LE:rough_potential_local}
Let $V\in C_0^{k,\kappa}(\R^d)$ be  real valued, where $k\in\N_0$ and $\kappa\in[0,1]$. Then for all $\varepsilon >0$ there exists a rough potential $V_\varepsilon \in C_0^{\infty}(\R^d)$ such that
\begin{equation}\label{EQ:rough_potential_local}
	\begin{aligned}
	 \big| \partial_x^\alpha V(x) -  \partial_x^\alpha V_\varepsilon(x)\big| \leq C_\alpha \varepsilon^{k+\kappa - |\alpha|} \quad\text{for all $\alpha\in\N_0^d$ such that $|\alpha|\leq k$}
	 \\
	  \big| \partial_x^\alpha V_\varepsilon(x)\big| \leq C_\alpha \varepsilon^{k+\kappa - |\alpha|} \quad\text{for all $\alpha\in\N_0^d$ such that $|\alpha|> k$},
	 \end{aligned}
\end{equation}
where the constants $C_\alpha$ are independent of $\varepsilon$ but depend on $\norm{\partial^\beta V}_{L^\infty(\R^d)}$ for $\beta\in\N_0^d$ with $|\beta|\leq \min(|\alpha|,k)$  and the H\"older constant for $V$. Moreover, if for some open set $\Omega$ and a constant $c>0$ it holds that
\begin{equation*}
	|V(x)| +\hbar^{\frac{2}{3}} \geq c \qquad\text{for all $x\in\Omega$}.
\end{equation*}
Then there exists a constant $\tilde{c}$ such that for all $\varepsilon$ sufficiently small it holds that
\begin{equation*}
	|V_\varepsilon(x)| +\hbar^{\frac{2}{3}} \geq \tilde{c} \qquad\text{for all $x\in \Omega$}.
\end{equation*}
\end{lemma}
\begin{proof}
A proof of the estimates in \eqref{EQ:rough_potential_local} can be found in either \cite[Proposition~1.1]{MR1974450} or \cite[Proposition 4.A.2]{ivrii2019microlocal1}. The second part of the lemma is a direct consequence of the estimates in \eqref{EQ:rough_potential_local}. To see this note that
\begin{equation}
	|V_\varepsilon(x) -V(x)| \leq C_0 \varepsilon^{k+\kappa} \implies |V_\varepsilon(x)| \geq  |V(x)| - C_0 \varepsilon^{k+\kappa}.
\end{equation}
Hence for $C_0\varepsilon^{k+\kappa} < \frac{c}{2}$ we obtain the desired estimate. This concludes the proof.
\end{proof}
We will in the following call the potentials depending on the parameter $\varepsilon$ for rough potentials. 
\begin{remark}\label{RE:approximations} \hfill \\
\begin{enumerate}
\item\label{RE:approximations_1}
Let $\mathcal{H}_{\hbar,\mu}$ be an operator acting in $L^2(\R^d)$ and assume it satisfies Assumption~\ref{Assumption:local_potential_1} with some open set $\Omega$, numbers $\hbar>0$, $\mu\geq0$ and $\gamma\in[0,1]$.  Whenever we have such an operator we have by assumption the associated magnetic Schr\"odinger operator $H_{\hbar,\mu}= (-i\hbar\nabla - \mu a)^2  + V$, where $V\in C_0^{5,\kappa}(\R^d)$. Applying Lemma~\ref{LE:rough_potential_local} to $V$ we can also associate the approximating rough Schr\"odinger operator $H_{\hbar,\mu,\varepsilon}= (-i\hbar\nabla - \mu a)^2  + V_\varepsilon$ to $\mathcal{H}_{\hbar,\mu}$. In what follows when we have an operator $\mathcal{H}_{\hbar,\mu}$ satisfies Assumption~\ref{Assumption:local_potential_1} we will just say with associated rough Schr\"odinger operator $H_{\hbar,\mu,\varepsilon}$. This will always be the operator we get from replacing $V$ by $V_\varepsilon$ from Lemma~\ref{LE:rough_potential_local}. 

\item \label{RE:approximations_2}
As mentioned in the introduction it would have been desirable not to assume that the magnetic vector potential is smooth. We could have used a version of Lemma~\ref{LE:rough_potential_local} to construct a smooth approximation under suitable regularity conditions. However, we will later use that $H_{\hbar,\mu} - H_{\hbar,\mu,\varepsilon}$ is a bounded operator, where we have used the same notation as above. Had we instead also replaced the magnetic vector potential with a smoothed out version and considered the operator $\tilde{H}_{\hbar,\mu,\varepsilon}= (-i\hbar\nabla - \mu a_\varepsilon)^2  + V_\varepsilon$ we would no longer have that $H_{\hbar,\mu} - \tilde{H}_{\hbar,\mu,\varepsilon}$ can be defined as a bounded operator.

\item\label{RE:approximations_3}
The non-critical condition introduced in Lemma~\ref{LE:rough_potential_local} is not the same as mentioned in the introduction. We have included a power of the semiclassical parameter so that the assumption is now the following 
\begin{equation}\label{EQ:RE:approximations_1}
	|V(x)| +\hbar^{\frac{2}{3}} \geq c \qquad\text{for all $x\in\Omega$}.
\end{equation}
Firstly if this assumption is met we can be in either of the following two cases
\begin{equation*}
	|V(x)|  \geq \frac{c}{2} \quad\text{for all $x\in\Omega$} \qquad\text{or}\qquad  \hbar^{\frac{2}{3}} \geq \frac{c}{2} \quad\text{for all $x\in\Omega$}.
\end{equation*}
If we are in the first case we have the usual non-critical condition. When we are in the second case, we have that the size of the semiclassical parameter is bounded from below. Since we by standard arguments can verify that all quantities of interest are finite, we can obtain the desired results by a suitable choice of constants. Hence we can make this more ``general''  non-critical assumption. 

The second thing is we have added $\hbar^{\frac{2}{3}}$ and not just $\hbar$. The above argument would also work in this case, however later we will do a scaling argument. When performing this scaling argument we need to ensure the smallest scale we are working on is of order $\hbar$. Since this is the smallest scale on which we can obtain favourable estimates. For this scaling argument we will have two functions $l(x)$ and $f(x)=\sqrt{l(x)}$ and we will need these to satisfy that
\begin{equation}\label{EQ:RE:approximations_2}
	l(x)f(x) \geq C \hbar,
\end{equation}
for all $x\in\R^d$ with some positive constant $C$. The exact choice for the function $l(x)$ will be proportional to \eqref{EQ:RE:approximations_1}. Hence we need the power of $\hbar$ appearing in \eqref{EQ:RE:approximations_1} to be less than or equal to $\frac{2}{3}$ to ensure the estimate in \eqref{EQ:RE:approximations_2}. The reason for choosing $\frac{2}{3}$ and not a smaller power is the identity
\begin{equation}\label{EQ:RE:approximations_2}
	(l(x)f(x))^{\frac{2}{3}} = l(x),
\end{equation}
where we have used how we choose the function $f$.
 \end{enumerate}
\end{remark}
\section{Rough $\hbar$-pseudo-differential operators}\label{sec:Rough_pseudo_diff_op} 
Our proof is based on the theory of $\hbar$-pseudo-differential operators ($\hbar$-$\Psi$DO's). To be precise we will need a rough version of the general theory. We will here recall properties and results concerning rough $\hbar$-$\Psi$DO's. A more complete discussion of these operators can be found in \cite{MR4689394}.  A version of rough $\hbar$-$\Psi$DO theory can be found in \cite{ivrii2019microlocal1}. It first appears in Vol. I Section 2.3.
\subsection{Definitions and basic properties}
By a rough pseudo-differential operator $A_\varepsilon(\hbar) = \OpW(a_\varepsilon)$ of regularity $\tau$ we mean the operator  
  \begin{equation}\label{EQ_def_PDO_op}
    \OpW(a_\varepsilon)\psi(x) = \frac{1}{(2\pi\hbar)^d} \int_{\R^{2d}}  e^{i\hbar^{-1} \langle x-y,p\rangle} a_\varepsilon( \tfrac{x +y}{2},p) \psi(y) \, d y \, d p \quad\text{for $\psi\in\mathcal{S}(\R^d)$},
  \end{equation}
where $a_\varepsilon(x,p)$ is a rough symbol of regularity $\tau \in \Z$ and satisfies for all  $\alpha,\beta\in\N^d_0$ that
\begin{equation}\label{symbolb_est_def}
	\begin{aligned}
  	|\partial_x^\alpha\partial_p^\beta a_\varepsilon(x,p)| \leq C_{\alpha\beta} \varepsilon^{\min(0,\tau-\abs{\alpha})} m (x,p) \quad\text{for all $(x,p)\in\R^d\times\R^d$}, 
    \end{aligned}
  \end{equation}
where $C_{\alpha\beta}$ is independent of $\varepsilon$ and $m$ is a tempered weight function. A tempered weight function is in some parts of the literature called an order function. The integral in \eqref{EQ_def_PDO_op} should be understood as an oscillatory integral. For $\varepsilon>0$, $\tau\in\Z$ and a tempered weight function $m$ we will use the notation $\Gamma_{\varepsilon}^{m,\tau}(\R^{2d})$ for the set of all $ a_\varepsilon(x,p)\in C^\infty(\R^{2d})$ which satisfies \eqref{symbolb_est_def} for all $\alpha,\beta\in\N^d_0$.

As we are interested in traces of our operators, it will be important for us to know, when the operator is bounded and trace class. This is the content of the following two theorems.

\begin{thm}\label{THM:cal-val-thm}
  Let $a_\varepsilon \in \Gamma_{\varepsilon}^{m,\tau}(\R^{2d})$, where we assume $m\in L^\infty(\R^{2d})$ and $\tau\geq0$. Suppose $\hbar\in(0,\hbar_0]$ and there exists a
  $\delta$ in $(0,1)$ such that $\varepsilon\geq\hbar^{1-\delta}$. Then there exists a constant
  $C_d$ and an integer $k_d$ only depending on the dimension such that
  \begin{equation*}
    \norm{\OpW(a_\varepsilon)\psi}_{L^2(\R^d)} \leq C_d \sup_{\substack{\abs{\alpha},\abs{\beta}\leq k_d \\ (x,p)\in \R^{2d}}} \varepsilon^{\abs{\alpha}} \abs{\partial_x^\alpha \partial_p^\beta a_\varepsilon(x,p)} \norm{\psi}_{L^2(\R^d)} \quad\text{for all $\psi\in\mathcal{S}(\R^d)$}.
  \end{equation*}
  Especially $\OpW(a_\varepsilon)$ can be extended to a bounded operator on $L^2(\R^d)$.
\end{thm}

\begin{thm}\label{THM:thm_est_tr}
 There exists a constant $C(d)$ only depending on the dimension such that
\begin{equation*}
    \norm{\OpW(a_\varepsilon)}_{\Tr} \leq  \frac{ C(d)} {\hbar^d} \sum_{\abs{\alpha}+\abs{\beta}\leq 2d+2}  \varepsilon^{\abs{\alpha}}   \hbar^{\delta\abs{\beta}}  \int_{\R^{2d}} |\partial_x^\alpha \partial_p^\beta a_\varepsilon(x,p)|  \,dxdp.
  \end{equation*}
  for every rough symbol $a_\varepsilon\in\Gamma_{\varepsilon}^{m,\tau}(\R^{2d})$ with
  $\tau\geq0$, $\hbar\in(0,\hbar_0]$ and $\varepsilon\geq\hbar^{1-\delta}$ for some $\delta\in(0,1)$.
\end{thm}
Both of these theorems can be found in \cite{MR4689394}, where they are Theorem~{3.25} and Theorem~3.26 respectively.
We will also need to calculate the trace of a rough  $\hbar$-$\Psi$DO. This is the content of the next theorem.
\begin{thm}\label{BTHM:trace_formula}
  Let $a_\varepsilon$ be in $\Gamma_{\varepsilon}^{m,\tau}(\R^{2d})$ with $\tau\geq0$ and suppose $\partial_x^\alpha \partial_p^\beta a_\varepsilon(x,p)$ is an
  element of $L^1(\R^{2d})$ for all $\abs{\alpha}+\abs{\beta}\leq 2d+2$. Then $\OpW(a_\varepsilon)$ is trace class and
  \begin{equation*}
    \Tr(\OpW(a_\varepsilon))=\frac{1}{(2\pi\hbar)^d} \int_{\R^{2d}} a_\varepsilon(x,p) \,dxdp.
  \end{equation*}	
\end{thm}
This theorem is Theorem~{3.27} from \cite{MR4689394}. We will also need to compose operators. The following theorem is a simplified version of Theorem~{3.24} from \cite{MR4689394} on composing rough $\hbar$-$\Psi$DO's.
\begin{thm}\label{THM:composition-weyl-thm}
Let $a_\varepsilon$ be in $\Gamma_{\varepsilon}^{m_1,\tau_{a}}(\R^{2d})$ and $b_\varepsilon$ be in $\Gamma_{\varepsilon}^{m_2,\tau_{b}}(\R^{2d})$ with $\tau_a,\tau_b\geq0$ and $m_1,m_2\in L^\infty(\R^{2d})$. Suppose $\hbar\in(0,\hbar_0]$ and $\varepsilon \geq \hbar^{1-\delta}$ for a $\delta \in(0,1)$ and let $\tau=\min(\tau_a,\tau_b)$. Then there exists a  a sequence of rough symbols in $\{c_{\varepsilon,j}\}_{j\in\N_0}$ such that $c_{\varepsilon,j} \in \Gamma_{\varepsilon}^{m_1 m_2,\tau-j}(\R^{2d})$ for all $j\in\N_0$ and for every $N\in\N$ there exists $N_\delta\geq N$ such that
  \begin{equation*}
    \OpW(a_\varepsilon) \OpW(b_\varepsilon)  =  \sum_{j=0}^{N_\delta} \hbar^j \OpW(c_{\varepsilon,j}) + \hbar^{N_\delta+1} \mathcal{R}_\varepsilon(N_\delta;\hbar),
  \end{equation*}
where $\mathcal{R}_\varepsilon(N_\delta;\hbar)$ is a rough $\hbar$-$\Psi$DO which satisfies the bound
\begin{equation}
	\hbar^{N_\delta+1} \norm{\mathcal{R}_\varepsilon(N_\delta;\hbar)}_{\mathrm{op}} \leq C_N \hbar^N,
\end{equation}
where $C_N$ is independent of $\varepsilon$, but depend on the numbers $N$, $\norm{m_1}_{L^\infty(\R^{2d})}$, $\norm{m_2}_{L^\infty(\R^{2d})}$ and the constants $C_{\alpha\beta}$ from \eqref{symbolb_est_def} for both $a_\varepsilon$ and $b_\varepsilon$. The rough symbols $c_{\varepsilon,j}$ are given by the formula
  \begin{equation*}
    c_{\varepsilon,j}(x,p) = (-i)^j \sum_{\abs{\alpha}+\abs{\beta}=j} \frac{1}{\alpha!\beta!}\Big(\frac{1}{2} \Big)^{\abs{\alpha}}\Big(-\frac{1}{2} \Big)^{\abs{\beta}} (\partial_p^\alpha \partial_x^\beta a_{\varepsilon})(x,p) (\partial_p^\beta \partial_x^\alpha b_{\varepsilon})(x,p).
  \end{equation*}
\end{thm}
\begin{remark}\label{Re:composition-weyl-thm}
Assume we are in the setting of Theorem~\ref{THM:composition-weyl-thm}. If we had assumed that at least one of the tempered weight functions $m_1$ or $m_2$ was in  $L^\infty(\R^{2d})\cap L^1(\R^{2d})$ we would get that the error term is not just bounded in operator norm but also in trace norm. That is
\begin{equation}
	\hbar^{N_\delta+1} \norm{\mathcal{R}_\varepsilon(N_\delta;\hbar)}_{1} \leq C_N \hbar^{N-d}.
\end{equation}
\end{remark}
The following lemma is an easy consequence of the result on the composition of $\hbar$-$\Psi$DO's. It can also be found in \cite{MR1343781} as Lemma~{2.2}.
\begin{lemma}\label{LE:disjoint_supp_PDO}
Let $\theta_1,\theta_2\in \mathcal{B}^\infty(\R^{2d})$ and suppose that there exists a constant $c>0$ such that
\begin{equation}\label{EQ:disjoint_supp_PDO}
	\dist(\supp(\theta_1),\supp(\theta_2)) \geq c.
\end{equation}
Then for all $N\in\N$, it holds that
\begin{equation*}
	\norm{\OpW(\theta_1) \OpW(\theta_2)}_{\mathrm{op}} \leq C_N \hbar^N.
\end{equation*}
If we further assume $\theta_1\in C_0^\infty(\R^{2d})$ it holds for all $N\in\N$ that
\begin{equation*}
	\norm{\OpW(\theta_1) \OpW(\theta_2)}_{1} \leq C_N \hbar^N.
\end{equation*}
In both cases the constant $C_N$ depends on the numbers $\norm{\partial^\alpha_x \partial^\beta_p \theta_1}_{L^\infty(\R^{2d})}$ and $\norm{\partial^\alpha_x \partial^\beta_p \theta_2}_{L^\infty(\R^{2d})}$ for all $\alpha,\beta\in\N_0^d$. In the second case, the constant $C_N$ will also depend on $c$ from \eqref{EQ:disjoint_supp_PDO}.
\end{lemma}
\subsection{Properties of rough Schr\"odinger operators}
We will in the following consider rough Schr\"odinger operators that satisfy the following assumption.
\begin{assumption}\label{Assumption:Rough_schrodinger}
Let $H_{\hbar,\mu,\varepsilon} =  (-i\hbar\nabla - \mu a)^2  + V_\varepsilon$ be a rough Schr\"odinger operator acting in $L^2(\R^d)$. Suppose that $a_j\in C_0^\infty(\R^d)$ for all $j\in\{1,\dots,d\}$ and are real valued. Moreover, suppose that $V_\varepsilon$ is a rough potential of regularity $\tau\geq0$ such that
\begin{enumerate}[label={$\roman*)$}]
  \item\label{R.P.1} $V_\varepsilon$ is real, smooth and $\min_{x\in\R^d} V_\varepsilon(x)>-\infty$.
  \item\label{R.P.2} There exists a $\zeta>0$ such that for all $\alpha\in\N_0^d$ there exists a constant $C_\alpha$ such that
    \begin{equation*}
     |\partial_x^\alpha V_\varepsilon(x)| \leq C_\alpha \varepsilon^{\min(0,\tau-|\alpha|)} (V_\varepsilon(x) +\zeta) \qquad\text{for all $x\in\R^d$}.
    \end{equation*}
  \item\label{R.P.3} There exists two constants $C,M>0$ such that
  \begin{equation*}
  	|V_\varepsilon(x)| \leq C(V_\varepsilon(y)+\zeta) (1+ |x-y|)^M \qquad\text{for all $x,y\in\R^d$}.
  \end{equation*}
  \end{enumerate}
\end{assumption}
\begin{remark}
When a rough Schr\"odinger operator $H_{\hbar,\mu,\varepsilon}$ satisfies Assumption~\ref{Assumption:Rough_schrodinger} it can be shown that as a $\hbar$-$\Psi$DO it is essentially self-adjoint. For details see e.g. \cite[Section 4]{MR4689394}. We will in these cases denote the closure by $H_{\hbar,\mu,\varepsilon}$ as well. 
In the case where  $H_{\hbar,\mu,\varepsilon} =  (-i\hbar\nabla - \mu a)^2  + V_\varepsilon$ with $a_j\in C_0^\infty(\R^d)$ for all $j\in\{1,\dots,d\}$ and $V_\varepsilon$ having compact support we have that $H_{\hbar,\mu,\varepsilon}$ satisfies Assumption~\ref{Assumption:Rough_schrodinger}.
\end{remark}
The following theorem is a simplified version of a more general theorem that can be found in \cite{MR4689394}.
\begin{thm}\label{THM:func_calc}
  Let $H_{\hbar,\mu,\varepsilon} =  (-i\hbar\nabla - \mu a)^2  + V_\varepsilon$ be a rough Schr\"odinger operator of regularity
  $\tau \geq 1$ acting in $L^2(\R^d)$ with $\hbar$ in $(0,\hbar_0]$ and $\mu\in[0,\mu_0]$.
  Suppose that $H_{\hbar,\mu,\varepsilon}$ satisfies
  Assumption~\ref{Assumption:Rough_schrodinger} and there exists a $\delta$
  in $(0,1)$ such that $\varepsilon\geq\hbar^{1-\delta}$. Then for any function $f\in C_0^\infty(\R)$ and every $N\in\N$ there exists a $N_\delta\in\N$ such that
    \begin{equation*}
    f(H_{\hbar,\mu,\varepsilon}) = \sum_{j = 0}^{N_\delta} \hbar^j \OpW(a_{\varepsilon,j}^f) + \hbar^{N_\delta+1} \mathcal{R}_\varepsilon(N_\delta,f;\hbar),
  \end{equation*}
  where
\begin{equation}
	\hbar^{N_\delta+1} \norm{\mathcal{R}_\varepsilon(N_\delta;\hbar)}_{\mathrm{op}} \leq C_N \hbar^N,
\end{equation}
  and 
  \begin{equation}\label{B.func_cal_sym} 
  	\begin{aligned}
  	a_{\varepsilon,0}^f(x,p) &= f( (p- \mu a(x))^2  + V_\varepsilon(x) ),
	\\
	a_{\varepsilon,1}^f(x,p) &= 0,
	\\
    	a_{\varepsilon,j}^f (x,p)&= \sum_{k=1}^{2j-1} \frac{(-1)^k}{k!} d_{\varepsilon,j,k}(x,p) f^{(k)}( (p- \mu a(x))^2  + V_\varepsilon(x)) \quad\quad\text{for $j\geq2$},
    \end{aligned}
  \end{equation}
   where $d_{\varepsilon, j ,k}$ are universal polynomials in
  $\partial_p^\alpha\partial_x^\beta [  (p- \mu a(x))^2  + V_\varepsilon(x)]$ for
  $\abs{\alpha}+\abs{\beta}\leq j$. Especially we have that $a_{\varepsilon,j}^f (x,p)$ is a rough symbol of regularity $\tau-j$ for all $j\in\N_0$.
\end{thm}
\begin{remark}\label{RE:propagator}
To prove the following theorem one will need to understand the Schr\"odinger propagator associated to $H_{\hbar,\mu,\varepsilon}$. That is the operator $e^{i\hbar^{-1}t H_{\hbar,\mu,\varepsilon}}$. Under the assumptions of the following theorem, we can find an operator with an explicit kernel that locally approximates $e^{i\hbar^{-1}t H_{\hbar,\mu,\varepsilon}}$ in a suitable sense. This local construction is only valid for times of order $\hbar^{1-\frac{\delta}{2}}$. But if we locally have a non-critical condition the approximation can be extended to a small time interval $[-T_0,T_0]$. For further details see \cite{MR4689394}. In the following, we will reference this remark and the number $T_0$. 
\end{remark}
\begin{thm}\label{THM:Expansion_of_trace}
 Let $H_{\hbar,\mu,\varepsilon}$ be a rough Schr\"odinger operator of regularity
  $\tau \geq 2$ acting in $L^2(\R^d)$ with $\hbar$ in $(0,\hbar_0]$ and $\mu\in[0,\mu_0]$ which  satisfies
  Assumption~\ref{Assumption:Rough_schrodinger}. Suppose there exists a $\delta$
  in $(0,1)$ such that $\varepsilon\geq\hbar^{1-\delta}$. Assume that $\theta\in C_0^\infty(\R^{2d})$ and there exists two constants  $\eta,c>0$ such that
  \begin{equation*}
  	|\nu -V_\varepsilon(x)| + \hbar^{\frac{2}{3}} \geq c \qquad\text{for all $(x,p)\in \supp(\theta)$ and $\nu\in(-2\eta,2\eta)$}.
  \end{equation*}
   Let $\chi$ be in $C^\infty_0((-T_0,T_0))$ and $\chi=1$ in a neighbourhood of $0$, where $T_0$ is the number from Remark~\ref{RE:propagator}. Then for every $f$ in $C_0^\infty((-\eta,\eta))$ we have that
  \begin{equation*}
   \Big|  \Tr\big[\OpW(\theta)f(H_{\hbar,\mu,\varepsilon})  \mathcal{F}_{\hbar}^{-1}[\chi](H_{\hbar,\mu,\varepsilon}-s)\big]  - \frac{1}{(2\pi\hbar)^{d}}  f(s)  \int_{\{a_{\varepsilon,0}=s\}} \frac{\theta }{\abs{\nabla{a_{\varepsilon,0}}}} \,dS_s\Big| \leq C\hbar^{2-d},
  \end{equation*}
  where $a_{\varepsilon,0}(x,p)=(p- \mu a(x))^2  + V_\varepsilon(x) $ and $S_s$ is the Euclidean surface measure on the surface $\{a_{\varepsilon,0}(x,p)=s\}$.  The error term is uniform with respect to
  $s \in (-\eta,\eta)$ but the constant $C$ depends on the dimension $d$, the numbers $\mu_0$, $\norm{\partial_x^\alpha\partial_p^\beta\theta}_{L^\infty(\R^d)}$ for all $\alpha,\beta\in N_0^d$,  $\norm{\partial^\alpha a_j}_{L^\infty(\R^d)}$ for all $\alpha\in\N_0^d$ and $j\in\{1\dots,d\}$,   $\norm{V}_{L^\infty(\supp(\theta))}$ and the numbers $C_\alpha$ from Assumption~\ref{Assumption:Rough_schrodinger}.
\end{thm}
This theorem is a special case of  \cite[Theorem~{6.1}]{MR4689394}. One thing to observe is that in the formulation of Theorem~{6.1} the assumption on the principal symbol $a_{\varepsilon,0}$ is 
  \begin{equation*}
    |\nabla_p a_{\varepsilon,0}(x,p)| \geq c \quad\text{for all } (x,p)\in a_{\varepsilon,0}^{-1}([-2\eta,2\eta]).
  \end{equation*}
This is technically the same assumption as the one we make in Theorem~\ref{THM:Expansion_of_trace} up to a square root. To see this note that we here have that $a_{\varepsilon,0}(x,p)= (p^2-\mu a(x))^2 + V_\varepsilon(x)$. Hence we have that
\begin{equation}\label{EQ:non_critial_assumption_com_1}
	 |\nabla_p a_{\varepsilon,0}(x,p)|^2 = 4(p-\mu a(x))^2 = 4(\nu -V_\varepsilon(x))
\end{equation} 
for all $(x,p)\in\R^{2d}$ such that  $a_{\varepsilon,0}(x,p)=\nu$. From \eqref{EQ:non_critial_assumption_com_1} we see that the two assumptions are indeed equivalent. Furthermore, if we had assumed the operator was of regularity $1$ we can obtain an error that is slightly better than $\hbar^{1-d}$ but not $\hbar^{2-d}$.

Before we continue we will need the following remark to set some notation and the following Proposition that is a type of Tauberian result.
\begin{remark}\label{RE:mollyfier_def}
Let $T\in(0,\min(T_0,T_1)]$, where $T_0$ is the number from Remark~\ref{RE:propagator} and $T_1$ is the number from Lemma~\ref{LE:localisation_propagator}. With this number $T$, let $\hat{\chi}\in C_0^\infty((-T,T))$ be a real valued function such that $\hat{\chi}(s)=\hat{\chi}(-s)$ and $\hat{\chi}(s)=1$ for all $t\in(-\frac{T}{2},\frac{T}{2})$. Define
\begin{equation*}
	\chi_1(t) = \frac{1}{2\pi} \int_{\R} \hat{\chi}(s) e^{ist} \,ds.
\end{equation*}
We assume that $\chi_1(t)\geq 0$ for all $t\in\R$ and there exist $T_3\in(0,T)$ and $c>0$ such that $\chi_1(t)\geq c$ for all $t\in[-T_3,T_3]$. We can guarantee these assumptions by (possible) replacing $\hat{\chi}$ by $\hat{\chi}*\hat{\chi}$. We will by $\chi_\hbar(t)$ denote the function
\begin{equation*}
	\chi_\hbar(t) = \tfrac{1}{\hbar} \chi_1(\tfrac{t}{\hbar}) = \mathcal{F}_\hbar^{-1}[\hat{\chi}](t).
\end{equation*}
Moreover for any function $g\in L^1_{loc}(\R)$ we will use the notation
\begin{equation*}
	g^{(\hbar)}(t) =g*\chi_\hbar(t) = \int_\R g(s) \chi_\hbar(t-s).
\end{equation*}
\end{remark}
\begin{proposition}\label{PRO:Tauberian}
Let $A$ be a self-adjoint operator acting in the Hilbert space $L^2(\R^d)$ and $g\in C_{0}^{\infty,\gamma}(\R)$. Let $\chi_1$ be defined as in Remark~\ref{RE:mollyfier_def}. If for a Hilbert-Schmidt operator $B$
\begin{equation}\label{EQ:PRO:Tauberian_1}
	\sup_{t\in \mathcal{D}(\delta)} \norm { B^{*}\chi_\hbar (A-t) B}_1 \leq Z(\hbar),
\end{equation}
where $\mathcal{D}(\delta) = \{ t \in\R \,|\, \dist(\supp(g)),t)\leq \delta \}$, $Z(\beta)$ is some positive function and $\delta$ is a strictly positive number. Then it holds that
\begin{equation}
	\norm { B^{*}(g(A)-g^{(\hbar)}(A)) B}_1 \leq C \hbar^{1+\gamma} Z(\hbar) + C_{N}' \hbar^N \norm{B^{*}B}_1 \quad\text{for all $N\in\N$},
\end{equation}
where the constants $C$ and $C'$ depend on the number $\delta$ and the functions $g$ and $\chi_1$ only.
\end{proposition}
The Proposition is taken from \cite{MR1343781}, where it is Proposition~{2.6}. It first appeared in \cite{MR1272980} for $\gamma\in(0,1]$. To apply this proposition
we will establish a case, where we have a bound of the type \eqref{EQ:PRO:Tauberian_1} from Proposition~\ref{PRO:Tauberian}.
\begin{lemma}\label{LE:verification_of_assumption}
 Let $H_{\hbar,\mu,\varepsilon}$ be a rough Schr\"odinger operator of regularity
  $\tau \geq 2$ acting in $L^2(\R^d)$ with $\hbar$ in $(0,\hbar_0]$ and $\mu\in[0,\mu_0]$ which  satisfies
  Assumption~\ref{Assumption:Rough_schrodinger}. Suppose there exists a $\delta$
  in $(0,1)$ such that $\varepsilon\geq\hbar^{1-\delta}$. Assume that $\theta\in C_0^\infty(\R^{2d})$ and there exists two constants  $\eta,c>0$ such that
  \begin{equation*}
  	|\nu -V_\varepsilon(x)| + \hbar^{\frac{2}{3}} \geq c \qquad\text{for all $(x,p)\in \supp(\theta)$ and $\nu\in(-2\eta,2\eta)$}.
  \end{equation*}
   Let $\chi_\hbar$ be the function from Remark~\ref{RE:mollyfier_def}. Then for every $f$ in $C_0^\infty((-\eta,\eta))$ we have that
  \begin{equation*}
  \norm{\OpW(\theta)f(H_{\hbar,\mu,\varepsilon}) \chi_\hbar(H_{\hbar,\mu,\varepsilon}-s) f(H_{\hbar,\mu,\varepsilon}) \OpW(\theta)}_1 \leq C\hbar^{-d},
  \end{equation*}
  where the constant depends on the dimension, the numbers $\mu_0$, $\norm{\partial_x^\alpha\partial_p^\beta\theta}_{L^\infty(\R^d)}$ for all $\alpha,\beta\in N_0^d$,  $\norm{\partial^\alpha a_j}_{L^\infty(\R^d)}$ for all $\alpha\in\N_0^d$ and $j\in\{1\dots,d\}$,   $\norm{V_\varepsilon}_{L^\infty(\supp(\theta))}$, the $C_\alpha$ for all $\alpha\in\N_0^d$ from Assumption~\ref{Assumption:Rough_schrodinger} and $\norm{\partial^\alpha f}_{L^\infty(\R^d)}$ for all $\alpha\in\N_0$.
\end{lemma} 
\begin{proof}
Since we assume that $\chi_\hbar(t)\geq0$ for all $t\in\R$ we have that the composition of the operators will be a positive operator and hence we have that
  \begin{equation}\label{EQ:verification_of_assumption_1}
  	\begin{aligned}
  \MoveEqLeft \norm{\OpW(\theta)f(H_{\hbar,\mu,\varepsilon}) \chi_\hbar(H_{\hbar,\mu,\varepsilon}-s) f(H_{\hbar,\mu,\varepsilon}) \OpW(\theta)}_1 
  \\
  &= \Tr[\OpW(\theta)f(H_{\hbar,\mu,\varepsilon}) \chi_\hbar(H_{\hbar,\mu,\varepsilon}-s) f(H_{\hbar,\mu,\varepsilon}) \OpW(\theta)] 
  \\
  &=  \Tr[\OpW(\theta)\OpW(\theta) f^2(H_{\hbar,\mu,\varepsilon}) \chi_\hbar(H_{\hbar,\mu,\varepsilon}-s) ], 
  \end{aligned}
  \end{equation}
where we in the last equality have used cyclicality of the trace. From applying Theorem~\ref{THM:composition-weyl-thm} and Remark~\ref{Re:composition-weyl-thm} we obtain that
  \begin{equation}\label{EQ:verification_of_assumption_2}
  	\begin{aligned}
	\MoveEqLeft  \big| \Tr[\OpW(\theta)\OpW(\theta) f^2(H_{\hbar,\mu,\varepsilon}) \chi_\hbar(H_{\hbar,\mu,\varepsilon}-s) ] \big| 
	\\
	&\leq  \big| \Tr[\OpW(\theta^2) f^2(H_{\hbar,\mu,\varepsilon}) \chi_\hbar(H_{\hbar,\mu,\varepsilon}-s) ]\big| +C\hbar^{-d}, 
  \end{aligned}
  \end{equation}
where the constant $C$ depends on the numbers $\norm{\partial^\alpha_x \partial^\beta_p \theta}_{L^\infty(\R^d)}$ and $\norm{f}_{L^\infty(\R)}$. Applying Theorem~\ref{THM:Expansion_of_trace} we get that
  \begin{equation}\label{EQ:verification_of_assumption_3}
  	\begin{aligned}
   \big| \Tr[\OpW(\theta^2) f^2(H_{\hbar,\mu,\varepsilon}) \chi_\hbar(H_{\hbar,\mu,\varepsilon}-s) ]\big| \leq C\hbar^{-d}, 
  \end{aligned}
  \end{equation}
where the constant $C$ depends on the dimension and the numbers $\mu_0$, $\norm{\partial_x^\alpha\partial_p^\beta\theta}_{L^\infty(\R^d)}$ for all $\alpha,\beta\in N_0^d$,  $\norm{\partial^\alpha a_j}_{L^\infty(\R^d)}$ for all $\alpha\in\N_0^d$ and $j\in\{1\dots,d\}$,   $\norm{V_\varepsilon}_{L^\infty(\supp(\theta))}$, the $C_\alpha$ for all $\alpha\in\N_0^d$ from Assumption~\ref{Assumption:Rough_schrodinger} and $\norm{\partial^\alpha f}_{L^\infty(\R^d)}$ for all $\alpha\in\N_0$. Finally by combining \eqref{EQ:verification_of_assumption_1}, \eqref{EQ:verification_of_assumption_2} and \eqref{EQ:verification_of_assumption_3} we obtain the desired estimate and this concludes the proof.
\end{proof}
\begin{thm}\label{THM:Loc_rough_prob}
Let $H_{\hbar,\mu,\varepsilon}$ be a rough Schr\"odinger operator of regularity $\tau \geq 2$ acting in $L^2(\R^d)$ with $\hbar$ in $(0,\hbar_0]$ and $\mu\in[0,\mu_0]$ which  satisfies Assumption~\ref{Assumption:Rough_schrodinger}. Suppose there exists a $\delta$ in $(0,1)$ such that $\varepsilon\geq\hbar^{1-\delta}$.
Moreover, suppose there exists some $c>0$ such that 
\begin{equation*}
	|V_\varepsilon(x)| +\hbar^{\frac{2}{3}} \geq c \qquad\text{for all $x\in \Omega$}.
\end{equation*}
Then for $\gamma\in[0,1]$ and any $g \in C^{\infty,\gamma}(\R)$ and any $\theta\in C_0^\infty(\Omega\times\R^d)$ it holds that
\begin{equation*}
	  \Big|\Tr[\OpW(\theta) g(H_{\hbar,\mu,\varepsilon})] - \frac{1}{(2\pi\hbar)^d} \int_{\R^{2d}}g \big((p- \mu a(x))^2+V_\varepsilon(x)\big)\theta(x,p)  \,dx dp \Big| \leq C \hbar^{1+\gamma-d},
\end{equation*}
where the constant $C$ is depending on the dimension, $\mu_0$, the numbers $\norm{\partial_x^\alpha\partial_p^\beta\theta}_{L^\infty(\R^d)}$ for all $\alpha,\beta\in N_0^d$,  $\norm{\partial^\alpha a_j}_{L^\infty(\R^d)}$ for all $\alpha\in\N_0^d$ and $j\in\{1\dots,d\}$,   $\norm{V_\varepsilon}_{L^\infty(\Omega)}$ and the numbers $\zeta$ and $C_\alpha$ for all $\alpha\in\N_0^d$ from Assumption~\ref{Assumption:Rough_schrodinger}.
\end{thm}
\begin{proof}
By continuity there exists an $\eta>0$ such that 
  \begin{equation*}
  	|\nu -V_\varepsilon(x)| + \hbar^{\frac{2}{3}} \geq \frac{c}{2} \qquad\text{for all $x\in \Omega$ and $\nu\in(-2\eta,2\eta)$}.
  \end{equation*}
Let $f_1,f_2 \in C_0^\infty(\R)$ such that $\supp(f_2)\subset (-\eta,\eta)$ and 
\begin{equation}\label{EQ:Loc_rough_prob_0}
	g(H_{\hbar,\mu,\varepsilon}) = f_1(H_{\hbar,\mu,\varepsilon}) +  f_2^2(H_{\hbar,\mu,\varepsilon})g(H_{\hbar,\mu,\varepsilon}).
\end{equation}
We can ensure this since $H_{\hbar,\mu,\varepsilon}$ is lower semibounded. With these functions, we have that
\begin{equation}\label{EQ:Loc_rough_prob_1}
	  \Tr[\OpW(\theta) g(H_{\hbar,\mu,\varepsilon})] =   \Tr[\OpW(\theta)  f_1(H_{\hbar,\mu,\varepsilon}) ] +  \Tr[\OpW(\theta) f_2^2(H_{\hbar,\mu,\varepsilon})g(H_{\hbar,\mu,\varepsilon})]. 
\end{equation}
We will consider each term separately and start by considering the first term on the right-hand side of \eqref{EQ:Loc_rough_prob_1}. Here we get by applying Theorem~\ref{THM:func_calc} that
\begin{equation}\label{EQ:Loc_rough_prob_2}
	  \Tr[\OpW(\theta)  f_1(H_{\hbar,\mu,\varepsilon}) ]  =   \Tr[\OpW(\theta)  \OpW(a^{f_1}_{\varepsilon,0}) ] + C\hbar^{2-d},
\end{equation}
where the constant $C$ depends on the numbers $\norm{\partial^\alpha_x \partial^\beta_p \theta}_{L^\infty(\R^d)}$  for all $\alpha,\beta\in\N_0^d$, $\norm{\partial^\alpha a_j}_{L^\infty(\R^d)}$ for all $\alpha\in\N^d$ and $j\in\{1\dots,d\}$ and $\norm{\partial_x^\alpha V_\varepsilon}_{L^\infty(\Omega)}$ for all $\alpha\in\N_0^d$. Moreover, we have used the notation $a_{\varepsilon,0}^{f_1}(x,p) = f_1( (p- \mu a(x))^2  + V_\varepsilon(x) )$. From applying Theorem~\ref{THM:composition-weyl-thm} and Theorem~\ref{BTHM:trace_formula} we  get that
\begin{equation}\label{EQ:Loc_rough_prob_3}
	\begin{aligned}
	\Tr[\OpW(\theta)  \OpW(a^{f_1}_{\varepsilon,0}) ] = {}& \frac{1}{(2\pi\hbar)^d} \int_{\R^{2d}} f_1\big( (p- \mu a(x))^2  + V_\varepsilon(x) \big) \theta(x,p) \,dxdp  
	\\
	&- \frac{i\hbar}{(2\pi\hbar)^d} \int_{\R^{2d}} c_{\varepsilon,1}(x,p) \,dxdp  +\mathcal{O}(\hbar^{2-d}),
	\end{aligned}
\end{equation}
where $c_{\varepsilon,1}$ is the subprincipal symbol we get from composing the operators. Since the left-hand side of \eqref{EQ:Loc_rough_prob_3} is real and $c_{\varepsilon,1}$ is real we have that the second term of the righthand side has to be of lower order. To see that the lefthand side is real note that for two self-adjoint operators $A$ and $B$, where one is trace-class and the other is bounded we have that
\begin{equation*}
	\overline{\Tr[AB]} = \Tr[B^{*}A^{*}] =  \Tr[AB]. 
\end{equation*}
In fact, one has, that the second term of the righthand side of \eqref{EQ:Loc_rough_prob_3} is zero. To see this one can calculate the coefficient and perform an integration by parts argument.   
Hence we have that
\begin{equation}\label{EQ:Loc_rough_prob_4}
	\begin{aligned}
	\Tr[\OpW(\theta)  \OpW(a^{f_1}_{\varepsilon,0}) ] = {}& \frac{1}{(2\pi\hbar)^d} \int_{\R^{2d}} f_1( (p- \mu a(x))^2  + V_\varepsilon(x) ) \theta(x,p) \,dxdp  +\mathcal{O}(\hbar^{2-d}).
	\end{aligned}
\end{equation}
Now we turn to the second term on the righthand side of \eqref{EQ:Loc_rough_prob_1}. When we consider this term we may, due to the support properties of $f_2$, assume that $\supp(g)\subset (-\frac{3}{2}\eta,0]$ that is $g\in C_0^{\infty,s}(\R)$. Let $g^\hbar$ be the smoothed version of $g$ as described in Remark~\ref{RE:mollyfier_def}. We then have that
\begin{equation}\label{EQ:Loc_rough_prob_5}
	\begin{aligned}
	   \Tr[\OpW(\theta) f_2^2(H_{\hbar,\mu,\varepsilon})g(H_{\hbar,\mu,\varepsilon})] = {}& \Tr[\OpW(\theta) f_2^2(H_{\hbar,\mu,\varepsilon})g^\hbar (H_{\hbar,\mu,\varepsilon})] 
	   \\
	   &+ \Tr\big[\OpW(\theta) f_2(H_{\hbar,\mu,\varepsilon}) [g (H_{\hbar,\mu,\varepsilon})-g^\hbar (H_{\hbar,\mu,\varepsilon})] f_2(H_{\hbar,\mu,\varepsilon})\big]. 
	   \end{aligned}
\end{equation}
Let  $\theta_1\in C_0^\infty(\Omega\times\R^d)$ such that $\theta \theta_1=\theta$. Then from applying Lemma~\ref{LE:disjoint_supp_PDO} twice we get  for all $N\in\N$ that
\begin{equation}\label{EQ:Loc_rough_prob_6}
	\begin{aligned}
	   \MoveEqLeft \big| \Tr\big[\OpW(\theta) f_2(H_{\hbar,\mu,\varepsilon}) [g(H_{\hbar,\mu,\varepsilon})-g^\hbar (H_{\hbar,\mu,\varepsilon})] f_2(H_{\hbar,\mu,\varepsilon})\big] \big| 
	   \\
	   &\leq \norm{\OpW(\theta)}_{\mathrm{op}} \big\lVert\OpW(\theta_1) f_2(H_{\hbar,\mu,\varepsilon}) [g (H_{\hbar,\mu,\varepsilon})-g^\hbar (H_{\hbar,\mu,\varepsilon})] f_2(H_{\hbar,\mu,\varepsilon})\OpW(\theta_1)\big\rVert_{1} + C_N \hbar^N.
	   \end{aligned}
\end{equation}
From Lemma~\ref{LE:verification_of_assumption} we have that assumption  \eqref{EQ:PRO:Tauberian_1} from Proposition~\ref{PRO:Tauberian} is satisfied with $B=f_2(H_{\hbar,\mu,\varepsilon})\OpW(\theta_1)$. Hence Proposition~\ref{PRO:Tauberian} gives us that
\begin{equation}\label{EQ:Loc_rough_prob_7}
	\begin{aligned}
	   \big\lVert\OpW(\theta_1) f_2(H_{\hbar,\mu,\varepsilon}) [g (H_{\hbar,\mu,\varepsilon})-g^\hbar (H_{\hbar,\mu,\varepsilon})] f_2(H_{\hbar,\mu,\varepsilon})\OpW(\theta_1)\big\rVert_{1} \leq C \hbar^{1+\gamma-d}.
	   \end{aligned}
\end{equation}
Using the definition of $g^\hbar$ and applying Theorem~\ref{THM:Expansion_of_trace} we have that
\begin{equation}\label{EQ:Loc_rough_prob_8}
	\begin{aligned}
	 \MoveEqLeft \Tr[\OpW(\theta) f_2^2(H_{\hbar,\mu,\varepsilon})g^\hbar (H_{\hbar,\mu,\varepsilon})]  
	 \\
	 ={}& \int_\R g(s)  \Tr[\OpW(\theta) f_2^2(H_{\hbar,\mu,\varepsilon})\chi_\hbar (H_{\hbar,\mu,\varepsilon}-s)] \,ds
	 \\
	 ={}& \frac{1}{(2\pi\hbar)^{d}} \int_\R g(s)   f_2^2(s)  \int_{\{a_{\varepsilon,0}=s\}} \frac{\theta }{\abs{\nabla{a_{\varepsilon,0}}}} \,dS_s \,ds + \mathcal{O}(\hbar^{2-d})
	 \\
	  ={}& \frac{1}{(2\pi\hbar)^{d}} \int_{\R^{2d}}  f_2^2g \big( (p- \mu a(x))^2  + V_\varepsilon(x) \big) \theta(x,p) \,dxdp + \mathcal{O}(\hbar^{2-d}).
	   \end{aligned}
\end{equation}
From combining \eqref{EQ:Loc_rough_prob_5}, \eqref{EQ:Loc_rough_prob_6}, \eqref{EQ:Loc_rough_prob_7} and \eqref{EQ:Loc_rough_prob_8} we obtain that
\begin{equation}\label{EQ:Loc_rough_prob_9}
	\begin{aligned}
	  \MoveEqLeft  \Tr[\OpW(\theta) f_2^2(H_{\hbar,\mu,\varepsilon})g(H_{\hbar,\mu,\varepsilon})] 
	  \\
	  = {}&  \frac{1}{(2\pi\hbar)^{d}} \int_{\R^{2d}}  f_2^2g\big( (p- \mu a(x))^2  + V_\varepsilon(x) \big) \theta(x,p) \,dxdp + \mathcal{O}(\hbar^{1+\gamma-d}).
	   \end{aligned}
\end{equation}
Recalling the identity in \eqref{EQ:Loc_rough_prob_0} and combining \eqref{EQ:Loc_rough_prob_1}, \eqref{EQ:Loc_rough_prob_4} and \eqref{EQ:Loc_rough_prob_9} we obtain that
\begin{equation*}
	  \Big|\Tr[\OpW(\theta) g(H_{\hbar,\mu,\varepsilon})] - \frac{1}{(2\pi\hbar)^d} \int_{\R^{2d}}g((p- \mu a(x))^2+V_\varepsilon(x))\theta(x,p)  \,dx dp \Big| \leq C \hbar^{1+\gamma-d},
\end{equation*}
where the constant depends on the numbers stated in the theorem. This concludes the proof.
\end{proof} 
\section{Auxiliary estimates}\label{sec:Aux_est}
We will in this section establish auxiliary estimates, which are needed for the proof of the main result. Especially we will prove bounds on $\norm{\varphi f(\mathcal{H}_{\hbar,\mu})}_1$, where $f\in C_0^\infty(\R)$, $\varphi\in C_0^\infty(\Omega)$ and $\mathcal{H}_{\hbar,\mu}$ satisfies Assumption~\ref{Assumption:local_potential_1} with some set $\Omega$ and the numbers $\hbar>0$ and $\mu\geq0$. Some of the results in this section are based on ideas originating in \cite{MR1343781}. The main estimate from which the other estimates are deduced is contained in the following Lemma. The Lemma is taken from \cite{MR1343781}, where it is Lemma~{3.6}. 
\begin{lemma}\label{LE:Resolvent_est_local}
Let $H_{\hbar,\mu} = (-i\hbar\nabla-\mu a)^2 +V$ be a Schr\"odinger operator acting in $L^2(\R^d)$ and assume that $V\in L^\infty(\R^d)$ and $a_j\in L^2_{loc}(\R^d)$ for $j\in\{1,\dots,d\}$.  Moreover, suppose that $\mu\leq\mu_0<1$ and $\hbar\in(0,\hbar_0]$, with $\hbar_0$ sufficiently small. Let $\varphi_1\in C^\infty_0(\R^d)$ and $\varphi_2\in \mathcal{B}^{\infty}(\R^d)$ such that
\begin{equation}
	\dist\big\{\supp(\varphi_1),\supp(\varphi_2)\big\} \geq c>0,
\end{equation}	
and let $r,m\in\{0,1\}$. Then for any $N>\frac{d}{2}$ it holds that
\begin{equation*}
	\norm{\varphi_1 Q_l^r  (H_\hbar - z)^{-1}  (Q_q^{*})^m \varphi_2}_1 \leq C_N \frac{\langle z \rangle^{\frac{m+r}{2}}}{|\im(z)|} \frac{\langle z \rangle^{\frac{d}{2}}}{\hbar^d} \frac{\langle z \rangle^N \hbar^{2N}}{|\im(z)|^{2N}},
\end{equation*}
where $Q_l = -i\hbar\partial_{x_l}-\mu a_l $.
The constant $C_N$ depends only on the numbers $N$, $\norm{\partial^\alpha\varphi_1}_{L^\infty(\R^d)}$, $\norm{\partial^\alpha\varphi_2}_{L^\infty(\R^d)}$ for all $\alpha\in\N_0^d$ and the constant $c$.
\end{lemma}
Note that we have phrased the lemma slightly differently here. In \cite{MR1343781} it is not stated with $|\im(z)|$ but with a function $d(z)$, that measures the distance from $z$ to the interval $[\lambda_0,\infty)$, where $\lambda_0 \leq \min(V(x)) -1$. We have done this since we will only need that $d(z)\geq |\im(z)|$.  
The next Lemma is also from \cite{MR1343781}, where it is Lemma~{3.9}.
\begin{lemma}\label{LE:Func_cal_est_inf_pon}
Let $H_{\hbar,\mu} = (-i\hbar\nabla-\mu a)^2 +V$ be a Schr\"odinger operator acting in $L^2(\R^d)$ and assume that $V\in L^\infty(\R^d)$ and $a_j\in L^2_{loc}(\R^d)$ for $j\in\{1,\dots,d\}$. Moreover, suppose that $\mu\leq\mu_0<1$ and $\hbar\in(0,\hbar_0]$, with $\hbar_0$ sufficiently small. Let $f\in C_0^\infty(\R)$ and $\varphi\in C_0^\infty(\R^d)$. Then
\begin{equation*}
	\norm{\varphi f(H_{\hbar,\mu})}_1 \leq C \hbar^{-d}.
\end{equation*}
 If $\varphi_1\in C^\infty_0(\R^d)$ and $\varphi_2\in \mathcal{B}^{\infty}(\R^d)$ such that
\begin{equation}
	\dist\big\{\supp(\varphi_1),\supp(\varphi_2)\big\} \geq c>0.
\end{equation}	
Then for any $N\geq0$, it holds that
\begin{equation*}
	\norm{\varphi_1  f(H_{\hbar,\mu}) \varphi_2}_1 \leq C_N  \hbar^N.
\end{equation*}
The constant $C_N$ depends on $\supp(f)$, the numbers $N$, $\norm{f}_{L^\infty(\R)}$, $\norm{\partial^\alpha\varphi_1}_{L^\infty(\R^d)}$, $\norm{\partial^\alpha\varphi_2}_{L^\infty(\R^d)}$ for all $\alpha\in\N_0^d$ and the constant $c$.
\end{lemma}
There is an almost similar result as the next lemma in \cite{MR1343781}. This is Theorem~{3.12}. The difference in the two results is that our constant will not be directly dependent on the number $\lambda_0$ (with the notation from \cite{MR1343781}). This is due to us using the Helffer-Sj\"ostrand formula instead of the representation formula for $f(A)$ used in \cite{MR1343781}, where $f\in C_0^\infty(\R^d)$ and $A$ is some self-adjoint lower semibounded operator.  
\begin{lemma}\label{LE:Comparision_Loc_infty} 
Let $\mathcal{H}_{\hbar,\mu}$ be an operator acting in $L^2(\R^d)$ which satisfies Assumption~\ref{Assumption:local_potential_1} with the open set $\Omega$ and the local operator $H_{\hbar,\mu} = (-i\hbar\nabla-\mu a)^2 +V $.  Assume that $\mu\leq\mu_0<1$ and $\hbar\in(0,\hbar_0]$, with $\hbar_0$ sufficiently small.
Then for $f\in C_0^\infty(\R)$ and $\varphi\in C_0^\infty(\Omega)$ we have for any $N\in\N_0$ that
\begin{equation*}
	\norm{\varphi[f(\mathcal{H}_{\hbar,\mu})-f(H_{\hbar,\mu})]}_1\leq C_N \hbar^N,
\end{equation*}	
and 
\begin{equation*}
	\norm{\varphi f(\mathcal{H}_{\hbar,\mu})}_1\leq C \hbar^{-d},
\end{equation*}	
The constant $C_N$ depends on $\supp(f)$ the numbers $N$, $\norm{f}_{L^\infty(\R)}$, $\norm{\partial^\alpha\varphi}_{L^\infty(\R^d)}$ for all $\alpha\in\N_0^d$.
\end{lemma}
\begin{proof}
Using the Helffer-Sj\"ostrand formula (Theorem~\ref{THM:Helffer-Sjostrand}) we obtain that
  \begin{equation}\label{EQ:Comparision_Loc_infty_1}
    \varphi[f(\mathcal{H}_{\hbar,\mu})-f(H_{\hbar,\mu})] =- \frac{1}{\pi} \int_\C   \bar{\partial }\tilde{f}(z) \varphi [(z-\mathcal{H}_{\hbar,\mu})^{-1}-(z-H_{\hbar,\mu})^{-1}] \, L(dz),
  \end{equation}
where $\tilde{f}$ is an almost analytic extension of $f$.
Since we assume that $\varphi\in C_0^\infty(\Omega)$ there exists a positive constant $c$ such that
\begin{equation*}
	\dist\big(\supp(\varphi),\partial\Omega\big) \geq 4c.
\end{equation*}
Let $\varphi_1 \in C_0^\infty(\R^d)$ such that $\varphi_1(x)\in[0,1]$ for all $x\in\R^d$. Moreover, we choose $\varphi_1$ such that $\varphi_1(x)=1$ on the set $\{x\in\R^d \,|\, \dist(\supp(\varphi),x) \leq c \}$ and 
\begin{equation*}
	\supp(\varphi_1)\subset \{x\in\R^d \,|\, \dist(\supp(\varphi),x) \leq 3c \}. 
\end{equation*}
With this function, we have that
  \begin{equation}\label{EQ:Comparision_Loc_infty_2}
     \begin{aligned}
     \MoveEqLeft \varphi [(z-\mathcal{H}_{\hbar,\mu})^{-1}-(z-H_{\hbar,\mu})^{-1}] 
     \\
     &=  \varphi [\varphi_1(z-\mathcal{H}_{\hbar,\mu})^{-1}-(z-H_{\hbar,\mu})^{-1}\varphi_1] -  \varphi (z-H_{\hbar,\mu})^{-1}(1-\varphi_1).
     \end{aligned}
  \end{equation}
For the second term on the right-hand side of \eqref{EQ:Comparision_Loc_infty_2} we have by Lemma~\ref{LE:Resolvent_est_local} for all $N>\frac{d}{2}$ that
\begin{equation}\label{EQ:Comparision_Loc_infty_3}
	\norm{\varphi (z-H_{\hbar,\mu})^{-1}(1-\varphi_1)}_1 \leq C_N \frac{\langle z \rangle^{N+\frac{d}{2}} \hbar^{2N-d}}{|\im(z)|^{2N+1}},
\end{equation}
where $C_N$ depends only on the numbers $N$, the functions $\varphi$, $\varphi_1$ and the constant $c$. For the first term on the right-hand side of \eqref{EQ:Comparision_Loc_infty_2} we have by the resolvent formalism that
  \begin{equation}\label{EQ:Comparision_Loc_infty_4}
     \begin{aligned}
     \MoveEqLeft  \varphi_1(z-\mathcal{H}_{\hbar,\mu})^{-1}-(z-H_{\hbar,\mu})^{-1}\varphi_1 
     \\
     &=   \sum_{j=1}^d (z-H_{\hbar,\mu})^{-1} [Q_j^{*}Q_j, \varphi_1 ] (z-\mathcal{H}_{\hbar,\mu})^{-1}
     \\
      &=  \sum_{j=1}^d (z-H_{\hbar,\mu})^{-1} \big(-i\hbar Q_j \partial_{x_j}\varphi_1  -\hbar^2\partial_{x_j}^2\varphi_1 \big) (z-\mathcal{H}_{\hbar,\mu})^{-1},
     \end{aligned}
  \end{equation}
where $\partial_{x_j}\varphi_1$ and $\partial_{x_j}^2\varphi_1$ are the derivatives of $\varphi_1$ with respect to $x_j$ once or twice respectively. Notice that due to our choice of $\varphi_1$, we have that 
\begin{equation*}
	\dist \big( \supp(\partial_{x_j}\varphi_1),\supp(\varphi) \big) \geq c \quad\text{and}\quad \dist \big( \supp(\partial_{x_j}^2\varphi_1),\supp(\varphi) \big) \geq c.
\end{equation*}
 Using \eqref{EQ:Comparision_Loc_infty_4} we have by Lemma~\ref{LE:Resolvent_est_local} for all $N>\frac{d}{2}$ that
  \begin{equation}\label{EQ:Comparision_Loc_infty_5}
     \begin{aligned}
     \MoveEqLeft \norm{ \varphi [\varphi_1(z-\mathcal{H}_{\hbar,\mu})^{-1}-(z-H_{\hbar,\mu})^{-1}\varphi_1] }_1
     \\
     &\leq  \big\lVert (z-\mathcal{H}_{\hbar,\mu})^{-1} \big\rVert_{\mathrm{op}}  \sum_{j=1}^d \hbar \big\lVert \varphi (z-H_{\hbar,\mu})^{-1}  Q_j \partial_{x_j}\varphi_1 \big\rVert_1  
     + \hbar^2 \big\lVert \varphi (z-H_{\hbar,\mu})^{-1}  \partial_{x_j}^2\varphi_1 \big\rVert_1 
     \\
     &\leq C_N \frac{\langle z \rangle^{N+\frac{d+1}{2}} \hbar^{2N-d}}{|\im(z)|^{2N+1}}\frac{\hbar+\hbar^2}{|\im(z)|},
     \end{aligned}
  \end{equation}
where $C_N$ depends only on the dimension the numbers $N$, the functions $\varphi$, $\varphi_1$ and the constant $c$. From combing  \eqref{EQ:Comparision_Loc_infty_2},  \eqref{EQ:Comparision_Loc_infty_3} and \eqref{EQ:Comparision_Loc_infty_5} we obtain that
  \begin{equation}\label{EQ:Comparision_Loc_infty_6}
     \big \lVert \varphi [(z-\mathcal{H}_{\hbar,\mu})^{-1}-(z-H_{\hbar,\mu})^{-1}] \big\rVert_1 \leq C_N \frac{\langle z \rangle^{N+\frac{d+1}{2}} \hbar^{2N-d}}{|\im(z)|^{2N+2}}.
  \end{equation}
Combining  \eqref{EQ:Comparision_Loc_infty_1}, \eqref{EQ:Comparision_Loc_infty_6} and using properties of the integral we get for all $N>\frac{d}{2}$ that 
  \begin{equation}\label{EQ:Comparision_Loc_infty_7}
  \begin{aligned}
      \big \lVert \varphi[f(\mathcal{H}_{\hbar,\mu})-f(H_{\hbar,\mu})]  \big \rVert_1 \ &\leq \frac{1}{\pi} \int_\C \big|  \bar{\partial }f(z) \big| \big \lVert \varphi [(z-\mathcal{H}_{\hbar,\mu})^{-1}-(z-H_{\hbar,\mu})^{-1}] \big\rVert_1 \, L(dz)
      \\
      &\leq  C_N \frac{\hbar^{2N-d}}{\pi} \int_\C   \big|\bar{\partial } f(z) \big| \frac{\langle z \rangle^{N+\frac{d+1}{2}} }{|\im(z)|^{2N+2}} \, L(dz)
      \leq \tilde{C}_N \hbar^{2N-d},
      \end{aligned}
  \end{equation}
where the constant $\tilde{C}_N$ depends on  the dimension the numbers $N$, the functions $\varphi$, $\varphi_1$, $f$ and the constant $c$. We have in the last inequality used the properties of the almost analytic extension $\tilde{f}$. The estimate in \eqref{EQ:Comparision_Loc_infty_7} concludes the proof.
\end{proof}
\begin{lemma}\label{LE:Func_com_model_reg}
Let $\mathcal{H}_{\hbar,\mu}$ be an operator acting in $L^2(\R^d)$. Suppose  $\mathcal{H}_{\hbar,\mu}$ satisfies Assumption~\ref{Assumption:local_potential_1} with the open set $\Omega$ and let   $H_{\hbar,\mu,\varepsilon} = (-i\hbar\nabla-\mu a)^2 +V_\varepsilon $ be the associated rough Schr\"odinger operator.  Assume that $\mu\leq\mu_0<1$ and $\hbar\in(0,\hbar_0]$, with $\hbar_0$ sufficiently small.
Let $f\in C_0^\infty(\R)$ and $\varphi\in C_0^\infty(\Omega)$ then it holds that
\begin{equation}\label{LE:EQ:Func_com_model_reg_2}
	\norm{\varphi [f(\mathcal{H}_{\hbar,\mu})-f(H_{\hbar,\mu,\varepsilon})]}_1 \leq C \varepsilon^{5+\kappa} \hbar^{-d}.
\end{equation}
The constant $C_N$ depends on $\supp(f)$ the dimension, the numbers $\norm{f}_{L^\infty(\R)}$, $\norm{\partial^\alpha\varphi}_{L^\infty(\R)}$ for all $\alpha\in\N^d_0$, $\norm{\partial^\alpha a_j}_{L^\infty(\R^d)}$ for all $\alpha\in\N_0^d$ with $|\alpha|\geq1$ and $j\in\{1\dots,d\}$ and $\norm{\partial_x^\alpha V}_{L^\infty(\R^d)}$ for all $\alpha\in\N_0^d$ such that $|\alpha|\leq 5$ and the H\"older constant for $V$.
\end{lemma}
\begin{proof}
Let $H_{\hbar,\mu}$ be the magenetic Schr\"odinger operator associated to $\mathcal{H}_{\hbar,\mu}$. We then have that
\begin{equation}\label{EQ:Func_com_model_reg_1}
	\norm{\varphi [f(\mathcal{H}_{\hbar,\mu})-f(H_{\hbar,\mu,\varepsilon})]}_1 \leq 	\norm{\varphi [f(\mathcal{H}_{\hbar,\mu})-f(H_{\hbar,\mu})]}_1  +  \norm{\varphi [f(H_{\hbar,\mu})-f(H_{\hbar,\mu,\varepsilon})]}_1.
\end{equation}
By Lemma~\ref{LE:Comparision_Loc_infty} it follows for all $N\in\N$ that
\begin{equation}\label{EQ:Func_com_model_reg_2}
	\norm{\varphi [f(\mathcal{H}_{\hbar,\mu})-f(H_{\hbar,\mu})]}_1 \leq C_N \hbar^N.
\end{equation}
To estimate the second term on the right-hand side of \eqref{EQ:Func_com_model_reg_1} let $f_1\in C_0^\infty(\R)$ such that $f_1(t)f(t)=f(t)$ for all $t\in \R$. Moreover, let $\varphi_1\in C_0^\infty (\Omega)$ such that $\varphi_1(x) = 1$ for all $x\in\overline{\supp(\varphi)}$. We then have for each $N\in\N_0$ that
\begin{equation}\label{EQ:Func_com_model_reg_3}
	\begin{aligned}
	\MoveEqLeft \norm{\varphi [f(H_{\hbar,\mu})-f(H_{\hbar,\mu,\varepsilon})]}_1 
	\\
	\leq{}& \norm{\varphi [f_1(H_{\hbar,\mu})-f_1(H_{\hbar,\mu,\varepsilon})] \varphi_1 f(H_{\hbar,\mu})}_1 +  \norm{\varphi f_1(H_{\hbar,\mu,\varepsilon})[f(H_{\hbar,\mu}) -f(H_{\hbar,\mu,\varepsilon})]}_1+C_N\hbar^N
	\\
	\leq{}&  C\hbar^{-d}\big[ \norm{f_1(H_{\hbar,\mu})-f_1(H_{\hbar,\mu,\varepsilon})}_{\mathrm{op}} + \norm{f(H_{\hbar,\mu})-f(H_{\hbar,\mu,\varepsilon})}_{\mathrm{op}}\big]+C_N\hbar^N,
	\end{aligned}
\end{equation}
where we have used Lemma~\ref{LE:Func_cal_est_inf_pon} three times. We can use this as $V,V_\varepsilon\in L^\infty(\R^d)$ and the functions $\varphi$ and $1-\varphi_1$ have disjoint support. Applying Theorem~\ref{THM:Helffer-Sjostrand} and the resolvent formalism we get that
\begin{equation}\label{EQ:Func_com_model_reg_4}
	\begin{aligned}
	  \norm{f(H_{\hbar,\mu})-f(H_{\hbar,\mu,\varepsilon})}_{\mathrm{op}} 
	  &\leq \frac{1}{\pi} \int_{\C} | \bar{\partial }\tilde{f}(z)| \norm{(H_{\hbar,\mu}-z)^{-1}-(H_{\hbar,\mu,\varepsilon} -z)^{-1}}_{\mathrm{op}} \, L(dz) 
	  \\
	  &\leq \frac{1}{\pi} \int_{\C}  \frac{| \bar{\partial }\tilde{f}(z)| }{|\im(z)|^2} \norm{V-V_{\varepsilon}}_{\mathrm{op}} \, L(dz) 
	  \\
	  &\leq C \varepsilon^{5+\kappa},
	\end{aligned}
\end{equation}
where we in the last inequality have used that $\tilde{f}$ is an almost analytic extension and have compact support and  Lemma~\ref{LE:rough_potential_local}. Analogously we obtain that
\begin{equation}\label{EQ:Func_com_model_reg_5}
	\begin{aligned}
	  \norm{f_1(H_{\hbar,\mu})-f_1(H_{\hbar,\mu,\varepsilon})}_{\mathrm{op}} 
	  \leq C \varepsilon^{5+\kappa}.
	\end{aligned}
\end{equation}
Combining the estimates in \eqref{EQ:Func_com_model_reg_1},  \eqref{EQ:Func_com_model_reg_2}, \eqref{EQ:Func_com_model_reg_3},  \eqref{EQ:Func_com_model_reg_4} and \eqref{EQ:Func_com_model_reg_5} we obtain the estimate in \eqref{LE:EQ:Func_com_model_reg_2}. This concludes the proof.
\end{proof}
Before we proceed we will need a technical Lemma. This Lemma gives us a version of the estimate \eqref{EQ:PRO:Tauberian_1} from Proposition~\ref{PRO:Tauberian}.  
\begin{lemma}\label{LE:assump_est_func_loc}
Let $H_{\hbar,\mu,\varepsilon} = (-i\hbar\nabla-\mu a)^2 +V_\varepsilon $ be a rough Schr\"odinger operator acting in $L^2(\R^d)$ of regularity $\tau\geq2$  with $\mu\leq\mu_0<1$ and $\hbar\in(0,\hbar_0]$, $\hbar_0$ sufficiently small. Assume that $a_j\in C_0^\infty(\R^d)$ for all $j\in\{1,\dots,d\}$ and $V_\varepsilon \in C_0^\infty(\R^d)$. Suppose there is an open set $\Omega \subset \supp(V_\varepsilon)$ and a $c>0$ such that 
\begin{equation*}
	|V_\varepsilon(x)| +\hbar^{\frac{2}{3}} \geq c \qquad\text{for all $x\in \Omega$}.
\end{equation*}
Let $\chi_\hbar(t)$ be the function from Remark~\ref{RE:mollyfier_def}, $f\in C_0^\infty(\R)$ and $\varphi\in C_0^\infty(\Omega)$ then it holds for $s\in\R$ that
\begin{equation*}
	\norm{\varphi f(H_{\hbar,\mu,\varepsilon}) \chi_\hbar(H_{\hbar,\mu,\varepsilon}-s)f(H_{\hbar,\mu,\varepsilon})  \varphi }_1  \leq C \hbar^{-d}.
\end{equation*}
The constant $C_N$ depends on $\supp(f)$, the dimension and the numbers $\norm{f}_{L^\infty(\R)}$, $\norm{\partial^\alpha\varphi}_{L^\infty(\R)}$ for all $\alpha\in\N^d_0$, $\norm{\partial^\alpha a_j}_{L^\infty(\R^d)}$ for all $\alpha\in\N_0^d$  and $j\in\{1\dots,d\}$,   $\norm{V_\varepsilon}_{L^\infty(\R^d)}$ and the numbers $C_\alpha$ from Assumption~\ref{Assumption:Rough_schrodinger}.
\end{lemma}
\begin{proof}
Under the assumptions of the lemma we have that $a$ and $V_\varepsilon$ satisfies Assumption~\ref{Assumption:Rough_schrodinger}.  Hence if we can find $\theta\in C_0^\infty(\Omega \times \R^d)$ such that for all $N\in\N$ we have that
\begin{equation}\label{EQ:assump_est_func_loc_1}
	\begin{aligned}
	\MoveEqLeft \big\lVert \varphi f(H_{\hbar,\mu,\varepsilon}) \chi_\hbar(H_{\hbar,\mu,\varepsilon}-s)f(H_{\hbar,\mu,\varepsilon})  \varphi  
	\\
	&- \varphi \OpW(\theta) f(H_{\hbar,\mu,\varepsilon}) \chi_\hbar(H_{\hbar,\mu,\varepsilon}-s)f(H_{\hbar,\mu,\varepsilon}) \OpW(\theta)  \varphi \big\rVert_1  \leq C_N \hbar^{N},
	\end{aligned}
\end{equation}
where $C_N$ has the dependency as stated in the lemma. Then the result will follow from Lemma~\ref{LE:verification_of_assumption}. Since this gives us that
\begin{equation*}
	\begin{aligned}
	\norm{\varphi \OpW(\theta) f(H_{\hbar,\mu,\varepsilon}) \chi_\hbar(H_{\hbar,\mu,\varepsilon}-s)f(H_{\hbar,\mu,\varepsilon}) \OpW(\theta)  \varphi }_1  \leq C \hbar^{-d}.
	\end{aligned}
\end{equation*}
In order to find such a $\theta$ we observe that since $V_\varepsilon$ and $a_j$ are bounded for all $j\in\{1,\dots,d\}$ there exist a $K>1$ such that
\begin{equation*}
	f(a_{\varepsilon,0}^f(x,p)) =0 \qquad\text{if $|p|\geq K-1$},
\end{equation*}
where we have also used that $f$ is compactly supported and the notation $a_{\varepsilon,0}^f(x,p) = f( (p- \mu a(x))^2  + V_\varepsilon(x) )$. Hence we will choose $\theta\in C_0^\infty(\Omega\times B(0,K+1))$ such that
\begin{equation*}
	\supp(\varphi)\cap \supp(1-\theta) \cap \supp(f(a_{\varepsilon,0}^f) )=\emptyset.
\end{equation*}
From applying Lemma~\ref{LE:disjoint_supp_PDO} and Theorem~\ref{THM:func_calc}  we obtain that
\begin{equation}\label{EQ:assump_est_func_loc_2}
	\begin{aligned}
	\norm{\varphi (1- \OpW(\theta)) f(H_{\hbar,\mu,\varepsilon}) }_{\mathrm{op}} \leq C_N\hbar^N.
	\end{aligned}
\end{equation}
By Theorem~\ref{THM:thm_est_tr} and Lemma~\ref{LE:Func_cal_est_inf_pon} we have that $\norm{\OpW(\theta)}_1\leq C\hbar^{-d}$ and  $\norm{\varphi f(H_{\hbar,\mu,\varepsilon}) }_{1}\leq C\hbar^{-d}$ respectively. Hence we get that
\begin{equation}
	\begin{aligned}
	\MoveEqLeft \big\lVert \varphi f(H_{\hbar,\mu,\varepsilon}) \chi_\hbar(H_{\hbar,\mu,\varepsilon}-s)f(H_{\hbar,\mu,\varepsilon})  \varphi  
	- \varphi \OpW(\theta) f(H_{\hbar,\mu,\varepsilon}) \chi_\hbar(H_{\hbar,\mu,\varepsilon}-s)f(H_{\hbar,\mu,\varepsilon}) \OpW(\theta)  \varphi \big\rVert_1
	\\
	\leq {}& \big\lVert \varphi (1-\OpW(\theta))  f(H_{\hbar,\mu,\varepsilon}) \chi_\hbar(H_{\hbar,\mu,\varepsilon}-s)f(H_{\hbar,\mu,\varepsilon})  \varphi  \big\rVert_1
	\\
	&+\big\lVert \varphi \OpW(\theta) f(H_{\hbar,\mu,\varepsilon}) \chi_\hbar(H_{\hbar,\mu,\varepsilon}-s)f(H_{\hbar,\mu,\varepsilon})(1- \OpW(\theta))  \varphi \big\rVert_1  
	\\
	\leq  {}&  C\hbar^{-d} \norm{\varphi (1- \OpW(\theta)) f(H_{\hbar,\mu,\varepsilon}) }_{\mathrm{op}} \leq C_N\hbar^N.
	\end{aligned}
\end{equation}
where we have used \eqref{EQ:assump_est_func_loc_2}. This establishes \eqref{EQ:assump_est_func_loc_1} and concludes the proof.
\end{proof}
In the same manner as the previous lemma, we will prove an asymptotic formula for the case with a compactly supported potential.
\begin{lemma}\label{LE:asymp_est_func_loc}
Let $H_{\hbar,\mu,\varepsilon} = (-i\hbar\nabla-\mu a)^2 +V_\varepsilon $ be a rough Schr\"odinger operator acting in $L^2(\R^d)$ of regularity $\tau\geq2$ with $\mu\leq\mu_0<1$ and $\hbar\in(0,\hbar_0]$, $\hbar_0$ sufficiently small. Assume that $a_j\in C_0^\infty(\R^d)$ for all $j\in\{1,\dots,d\}$ and $V_\varepsilon \in C_0^\infty(\R^d)$. Suppose there is an open set $\Omega \subset \supp(V_\varepsilon)$ and a $c>0$ such that 
\begin{equation*}
	|V_\varepsilon(x)| +\hbar^{\frac{2}{3}} \geq c \qquad\text{for all $x\in \Omega$}.
\end{equation*}
Then for $g\in C^{\infty,\gamma}(\R)$ with $\gamma\in[0,1]$ and any $\varphi\in C_0^\infty(\Omega)$ it holds that
\begin{equation*}
	  \Big|\Tr[\varphi g(H_{\hbar,\mu,\varepsilon})] - \frac{1}{(2\pi\hbar)^d} \int_{\R^{2d}}g(p^2+V_\varepsilon(x))\varphi(x)  \,dx dp \Big| \leq C \hbar^{1+\gamma-d}.
\end{equation*}
The constant $C_N$ depends on $\supp(f)$, the dimension and the numbers $\norm{f}_{L^\infty(\R)}$, $\norm{\partial^\alpha\varphi}_{L^\infty(\R)}$ for all $\alpha\in\N^d_0$, $\norm{\partial^\alpha a_j}_{L^\infty(\R^d)}$ for all $\alpha\in\N_0^d$ with $|\alpha|\geq1$ and $j\in\{1\dots,d\}$,   $\norm{V_\varepsilon}_{L^\infty(\R^d)}$ and the numbers $C_\alpha$ from Assumption~\ref{Assumption:Rough_schrodinger}.
\end{lemma}
\begin{proof}
Since all quantities are Gauge invariant we can start by performing a Gauge transform such that the supremums norm of the magnetic vector potential is uniformly bounded. 

As in the proof of Lemma~\ref{LE:assump_est_func_loc} we let $\theta\in C_0^\infty(\Omega\times B(0,K+1))$ such that
\begin{equation*}
	\supp(\varphi)\cap \supp(1-\theta) \cap \supp(f(a_{\varepsilon,0}^f) )=\emptyset,
\end{equation*}
where $a_{\varepsilon,0}^f(x,p) = f( (p- \mu a(x))^2  + V_\varepsilon(x) )$. Then as in the proof of Lemma~\ref{LE:assump_est_func_loc} we then get for all $N\in\N$ that
\begin{equation}\label{EQ:asymp_est_func_loc_1}
	  \Tr[\varphi g(H_{\hbar,\mu,\varepsilon})] =  \Tr[\varphi \OpW(\theta) g(H_{\hbar,\mu,\varepsilon})]  + \mathcal{O}(\hbar^N).
\end{equation}
This choice of $\theta$ ensures the assumptions of Theorem~\ref{THM:Loc_rough_prob} is satisfied. Hence we get that
\begin{equation}\label{EQ:asymp_est_func_loc_2}
	  \Big|\Tr[\varphi \OpW(\theta) g(H_{\hbar,\mu})] - \frac{1}{(2\pi\hbar)^d} \int_{\R^{2d}}g((p- \mu a(x))^2+V_\varepsilon(x))\varphi(x)\theta(x,p)  \,dx dp \Big| \leq C \hbar^{1+\gamma-d},
\end{equation}
From the support properties of $\theta$ we have that
\begin{equation}\label{EQ:asymp_est_func_loc_3}
	   \int_{\R^{2d}}g((p- \mu a(x))^2+V_\varepsilon(x))\varphi(x)\theta(x,p)  \,dx dp =    \int_{\R^{2d}}g(p^2+V_\varepsilon(x))\varphi(x)  \,dx dp.
\end{equation}
From combining \eqref{EQ:asymp_est_func_loc_1}, \eqref{EQ:asymp_est_func_loc_2} and \eqref{EQ:asymp_est_func_loc_3} we obtain the desired estimate. This concludes the proof.
\end{proof}
To compare the trace norm of certain operators we will need the following technical lemma. The lemma is an ``extension'' to Lemma~\ref{LE:Func_cal_est_inf_pon} under additional regularity assumptions on the potential. 

\begin{lemma}\label{LE:localisation_propagator}
Let $\gamma\in[0,1]$ and $H_{\hbar,\mu} = (-i\hbar\nabla-\mu a)^2 +V$ be a Schr\"odinger operator acting in $L^2(\R^d)$ and assume that $a_j\in \mathcal{B}^\infty (\R^d)$ for $j\in\{1,\dots,d\}$ and that $V\in C_0^{k,\kappa}(\R^d)$ with $k\in\N$ and $\kappa\in(0,1]$ such that $k=5$ and $\kappa>\gamma$ if $\gamma<1$ and $k=6$ and $\kappa>0$ if $\gamma=1$.  Moreover, suppose that $\mu\leq\mu_0<1$ and $\hbar\in(0,\hbar_0]$, with $\hbar_0$ sufficiently small. Let $\theta_1\in C^\infty_0(\R^{2d})$ and $\theta_2\in \mathcal{B}^{\infty}(\R^{2d})$ such that
\begin{equation}\label{support_conditions}
	\dist\big\{\supp(\theta_1),\supp(\theta_2)\big\} \geq c>0,
\end{equation}	
and let $f\in C_0^\infty(\R)$. Then there exists $T_1>0$ sufficiently small such that 
\begin{equation*}
	\norm{\OpW(\theta_2) e^{it\hbar^{-1} H_{\hbar,\mu} } f(H_{\hbar,\mu}) \OpW(\theta_1)}_{\mathrm{op}} \leq C \hbar^{3 +\gamma},
\end{equation*}
uniformly for $t\in[-T_1,T_1]$.
The constant $C$ depends on $\supp(f)$, the dimension, the numbers $N$, $\norm{f}_{L^\infty(\R^d)}$, $\norm{\partial^\alpha\varphi_1}_{L^\infty(\R^d)}$, $\norm{\partial^\alpha\varphi_2}_{L^\infty(\R^d)}$ for all $\alpha\in\N_0^d$ and the constant $c$.
\end{lemma}
The lemma can be interpreted as a consequence of Bohr's correspondence principle. Indeed, if we have two sets in phase space that satisfy \eqref{support_conditions} and one is bounded. Then we could consider a particle starting from the bounded set. Then for any classical Hamiltonian flow, there will be a time $T$ such that if $t\leq T$ the particle will not have entered the second set. This is the intuition behind the above lemma. There is also a version of this result in \cite[Lemma 5.1]{MR1272980} for $\hbar$-$\Psi$DO's. 
\begin{proof}
Firstly, we note that in the case $t=0$ the result follows from Lemma~\ref{LE:Func_cal_est_inf_pon}. Hence we may assume that $t\neq0$. We will in the following assume $0<t\leq T_1$. The case $-T_1\geq t >0$ is proven analogously.
We let $\psi\in C_0^\infty(\R)$ such that $\psi(s)=1$ for $|s|\leq1$ and $\psi(s)=0$ for $|s|\geq2$. Moreover, let $M$ be a sufficiently large constant which will be fixed later
and put
\begin{equation*}
	\psi_{\mu_1}(z) = \psi\Big(\tfrac{\im(z)}{\mu_1}\Big),
\end{equation*}
where $\mu_1=M\hbar\log(\frac{1}{\hbar})$. Moreover, we also let $\tilde{f}$ be an almost analytic extension of $f$ as described in Definition~\ref{Def:almost_analytic_function}. For the product $\tilde{f}\psi_{\mu_1}$ we have that
\begin{equation}\label{EQ:localisation_propagator_1}
	|\bar{\partial}(\tilde{f}\psi_{\mu_1})| 
	\leq  C \mu_1^{-1} \boldsymbol{1}_{[-2,-1]}(\tfrac{\im(z)}{\mu_1}) +
	C_n \psi_{\mu_1}(z) \abs{\im(z)}^n + C\mu_1^{-1} \boldsymbol{1}_{[1,2]}(\tfrac{\im(z)}{\mu_1}) 
\end{equation}
for all $n\in\N$, where we have used the properties of $\tilde{f}$. Since we have that $e^{it\hbar^{-1}z}$ is holomorphic we get the following identity from the Helffer-Sj\"ostrand formula (Theorem~\ref{THM:Helffer-Sjostrand})
\begin{equation}\label{EQ:localisation_propagator_2}
	\begin{aligned}
	\MoveEqLeft \OpW(\theta_2) e^{it\hbar^{-1} H_{\hbar,\mu} } f(H_{\hbar,\mu}) \OpW(\theta_1)
	\\
	&= -\frac{1}{\pi} \int_{\C} \bar{\partial}[\tilde{f}\psi_{\mu_1}](z)e^{it\hbar^{-1}z} \OpW(\theta_2) (z-H_{\hbar,\mu})^{-1} \OpW(\theta_1) \,L(dz),
	\end{aligned}
\end{equation}
where we have used the linearity of the integral. Combining the estimate in \eqref{EQ:localisation_propagator_1} and the identity in \eqref{EQ:localisation_propagator_2} we obtain the estimate
\begin{equation}\label{EQ:localisation_propagator_3}
	\begin{aligned}
	\MoveEqLeft \lVert \OpW(\theta_2) e^{it\hbar^{-1} H_{\hbar,\mu} } f(H_{\hbar,\mu}) \OpW(\theta_1)  \rVert_{\mathrm{op}} 
	\\
	\leq{}& \frac{C}{\mu_1} \int_{\substack{-\eta <\re(z)
        <\eta \\ -2\mu_1\leq \im(z)\leq-\mu_1}} e^{- t \hbar^{-1} \im(z)} \big\lVert \OpW(\theta_2)(z-H_{\hbar,\mu})^{-1} \OpW(\theta_1) \big\rVert_{\mathrm{op}} \,L(dz)
        \\
	&+ C_n \int_{\substack{-\eta <\re(z)
        <\eta \\ -2\mu_1\leq \im(z)\leq2\mu_1}} \abs{\im(z)}^n  e^{- t \hbar^{-1} \im(z)} \big\lVert \OpW(\theta_2) (z-H_{\hbar,\mu})^{-1} \OpW(\theta_1)\big\rVert_{\mathrm{op}} \,L(dz)
         \\
	&+ \frac{C}{\mu_1} \int_{\substack{-\eta <\re(z)
        <\eta \\ \mu_1\leq \im(z)\leq2 \mu_1}}  \big\lVert \OpW(\theta_2) (z-H_{\hbar,\mu})^{-1}\OpW(\theta_1) \big\rVert_{\mathrm{op}} \,L(dz),
        \end{aligned}
\end{equation}
where we have used that $\tilde{f}$ is compactly supported to ensure that we only integrate over a finite interval for $\re(z)$.  For the second term on the right-hand side, we will use the simple estimate
\begin{equation}\label{EQ:localisation_propagator_4}
	\big\lVert\OpW(\theta_2) (z-H_{\hbar,\mu})^{-1} \OpW(\theta_1) \big\rVert_{\mathrm{op}} \leq \frac{C}{|\im(z)|}.
\end{equation}
Applying this estimate we obtain for $n$ sufficiently large depending on $T_1$ and $M$ that
\begin{equation}\label{EQ:localisation_propagator_5}
	\begin{aligned}
	\MoveEqLeft  C_n \int_{\substack{-\eta <\re(z)<\eta \\ -2\mu_1\leq \im(z)\leq2\mu_1}} \abs{\im(z)}^n  e^{- t \hbar^{-1} \im(z)} \big\lVert\OpW(\theta_2) (z-H_{\hbar,\mu})^{-1} \OpW(\theta_1) \big\rVert_{\mathrm{op}} \,L(dz)
         \\
	\leq{}& \tilde{C}_n \int_{\substack{-\eta <\re(z)<\eta \\ -2\mu_1\leq \im(z)\leq2\mu_1}} \abs{\mu_1}^{n-1}  e^{2 t \hbar^{-1} \mu_1} \,L(dz) 
	\\
	\leq{}& C  \mu_1^{n}  e^{2 T_1 \hbar^{-1} \mu_1}
	\leq C \big(M\hbar\log(\tfrac{1}{\hbar})\big)^n \hbar^{-2T_1M} \leq C\hbar^{3+\gamma}.
        \end{aligned}
\end{equation}
To estimate the two remaining terms we cannot just use the simple estimate in \eqref{EQ:localisation_propagator_4}, as this will not give the desired result. Instead, we need to utilise that $\theta_1$ and $\theta_2$ have disjoint supports. This argument relies on pseudo-differential techniques. Hence we will first change our potential. Let $V_\varepsilon$ be the rough potential associated with $V$ as constructed in Lemma~\ref{LE:rough_potential_local} and let $H_{\hbar,\mu,\varepsilon}$ be the associated rough Sch\"odinger operator. We choose $\varepsilon=\hbar^{1-\delta}$, where 
\begin{equation*}
	\delta= \begin{cases}
	1-\frac{5+\gamma + \frac{1}{2}(\kappa-\gamma)}{5+\kappa} & \text{if $\gamma<1$}
	\\
	1-\frac{6+ \frac{1}{2}\kappa}{6+\kappa} & \text{if $\gamma=1$}.
	\end{cases} 
\end{equation*}
We then have that
\begin{equation*}
	\big\lVert\OpW(\theta_2)\big[ (z-H_{\hbar,\mu})^{-1} - (z-H_{\hbar,\mu,\varepsilon})^{-1} \big] \OpW(\theta_1) \big\rVert_{\mathrm{op}} \leq \frac{C \norm{V -V_\varepsilon}_{\mathrm{op}}}{|\im(z)|^2} \leq \frac{C\varepsilon^{k+\kappa}}{|\im(z)|^2}.
\end{equation*}
With our choice of $\varepsilon$ and $\delta$, this estimate implies that
\begin{equation}\label{EQ:localisation_propagator_6}
	\begin{aligned}
	\MoveEqLeft \big\lVert\OpW(\theta_2)(z-H_{\hbar,\mu})^{-1}\OpW(\theta_1) \big\rVert_{\mathrm{op}} 
	\\
	&\leq \big\lVert\OpW(\theta_2) (z-H_{\hbar,\mu,\varepsilon})^{-1} \OpW(\theta_1) \big\rVert_{\mathrm{op}}+ \frac{C\hbar^{5+\gamma+\frac{1}{2}(\kappa-\gamma')}}{|\im(z)|^2},
	\end{aligned}
\end{equation}
where $\gamma' = \gamma \boldsymbol{1}_{[0,1)}(\gamma)$. To estimate the operator norm on the right-hand side of \eqref{EQ:localisation_propagator_6} 
we will in the following let $\tilde{G}\in C_0^\infty(\R^{2d})$ such that $\tilde{G}\theta_1 = \theta_1$ and  
\begin{equation*}
	\dist\big\{\supp(\tilde{G}),\supp(\theta_2)\big\} \geq \tilde{c}>0.
\end{equation*}
For any $\nu>0$ we set $G=\nu \tilde{G}$. Moreover, we also define the rough pseudo-differential operators $\OpW(e^{G(x,p)\log(\hbar^{-1})})$. This operator will be rough in the sense of the monographs \cite{MR1735654,MR2952218}. It is defined by
\begin{equation*}
	\begin{aligned}
	\OpW(e^{G(x,p)\log(\hbar^{-1})}) \varphi (x) &= \frac{1}{(2\pi\hbar)^d} \int_{\R^{2d}} e^{i\hbar^{-1} \langle x-y,p\rangle } e^{G(\frac{x+y}{2},p)\log(\hbar^{-1})} \varphi(y) \,dydp,
	\end{aligned}
\end{equation*}
for all $\varphi\in \mathcal{S}(\R^d)$. It extends to a bounded operator on $L^2(\R^d)$ and it is elliptic and hence invertible if $\hbar\in(0,\hbar_0]$, where $\hbar_0>0$ is small enough. We will by $\OpW(e^{G(x,p)\log(\hbar^{-1})})^{-1}$ denote the inverse. The inverse will be an admissible operator in the sense that for every $N\in\N$ there exists $\N_0\in\N$, a sequence of symbols $\{a_j(x,p,\hbar)\}_{j=0}^{N_0}$ in $\mathcal{B}^\infty(\R^{2d})$ and a bounded operator $R_{N_0}(\hbar)$ such that 
\begin{equation*}
	\OpW(e^{G(x,p)\log(\hbar^{-1})})^{-1} = \sum_{j=0}^{N_0} \hbar^j \OpW(a_j) + \hbar^{N} R_{N_0}(\hbar),
\end{equation*}
and
\begin{equation*}
	\sup_{\hbar\in(0,1]} \norm{ R_{N_0}(\hbar)}_{\mathrm{op}} <\infty.
\end{equation*}
The principal symbol $a_0(x,p,\hbar)$  of $\OpW(e^{G(x,p)\log(\hbar^{-1})})^{-1}$ is given by the expression
\begin{equation*}
	a_0(x,p,\hbar) = e^{-G(x,p)\log(\hbar^{-1})}. 
\end{equation*}
For the remaining symbols $a_j(x,p,\hbar)$ of $\OpW(e^{G(x,p)\log(\hbar^{-1})})^{-1}$ we have that the subprincipal symbol $a_1(x,p,\hbar)=0$. For $j\geq2$ we have that
\begin{equation*}
	a_j(x,p,\hbar) = \tilde{a}_j(x,p,\hbar) e^{-G(x,p)\log(\hbar^{-1})},
\end{equation*}
where $\tilde{a}_j(x,p,\hbar)\in C_0^\infty(\R^{2d})$ such that $\supp(\tilde{a}_j)\cap\supp(\theta_i)=\emptyset$ for $i=1,2$. The symbols $\tilde{a}_j(x,p,\hbar)$ will be polynomials in $\partial_p^\alpha \partial_x^\beta \tilde{G}(x,p)$ and satisfy the bounds
\begin{equation*}
	\big|\partial_p^\alpha \partial_x^\beta a_j(x,p,\hbar) \big| \leq \log(\hbar^{-1})^{2j} C_{\alpha\beta},
\end{equation*}
for all $\alpha,\beta \in \N_0^d$. Hence applying Theorem~{4.23} from \cite{MR2952218} it follows that
\begin{equation}\label{EQ:localisation_propagator_6.5}
	\lVert \OpW(e^{G(x,p)\log(\hbar^{-1})})^{-1} \rVert_{\mathrm{op}} \leq C \qquad\text{and}\qquad \lVert \OpW(e^{G(x,p)\log(\hbar^{-1})}) \rVert_{\mathrm{op}} \leq C \hbar^{-\nu},
\end{equation}
where the constants $C$ depends on the dimension and the function $\tilde{G}$. Using the standard theory of rough pseudo-differential operators we obtain that
\begin{equation}\label{EQ:localisation_propagator_7}
	\begin{aligned}
	 \OpW(e^{G(x,p)\log(\hbar^{-1})})^{-1}(z-H_{\hbar,\mu,\varepsilon}) \OpW(e^{G(x,p)\log(\hbar^{-1})}) 
	=z-H_{\hbar,\mu,\varepsilon} + \nu \hbar \log(\hbar^{-1})  \mathcal{R}_1(\hbar),
	\end{aligned}
\end{equation}
where the error operator $\mathcal{R}_1(\hbar)$ satisfies that
\begin{equation}\label{EQ:localisation_propagator_7.5}
	\begin{aligned}
	\sup_{\hbar\in(0,\hbar_0]} \lVert \mathcal{R}_1(\hbar) \rVert_{\mathrm{op}} \leq C.
	\end{aligned}
\end{equation}
The constant $C$ depends on the dimension and the function $\tilde{G}$. This implies that for all $\hbar$ sufficiently small (depending on $\nu$) we have that the left-hand side of \eqref{EQ:localisation_propagator_7} is invertible. We get for all $\hbar>0$ sufficiently small that 
 \begin{equation}\label{EQ:localisation_propagator_8}
 	\begin{aligned}
	& e^{\nu  \log(\hbar^{-1})} \big\lVert \OpW(\theta_2) (z-H_{\hbar,\mu,\varepsilon})^{-1} \OpW(\theta_1)\big\rVert_{\mathrm{op}}
	\\
	&\leq   \big\lVert \OpW(\theta_2) \OpW(e^{G(x,p)\log(\hbar^{-1})})^{-1} (z-H_{\hbar,\mu,\varepsilon})^{-1} \OpW(e^{G(x,p)\log(\hbar^{-1})})   \OpW(\theta_1)\big\rVert_{\mathrm{op}} +   \frac{C_{\tilde{N}} \hbar^{\tilde{N}} }{|\im(z)|}
	\\
	&=   \big\lVert \OpW(\theta_2) \big[z- \OpW(e^{G(x,p)\log(\hbar^{-1})})^{-1} H_{\hbar,\mu,\varepsilon} \OpW(e^{G(x,p)\log(\hbar^{-1})}) \big]^{-1}  \OpW(\theta_1)\big\rVert_{\mathrm{op}} +   \frac{C_{\tilde{N}} \hbar^{\tilde{N}} }{|\im(z)|}
	\\
	&\leq \frac{ C+ C_{\tilde{N}} \hbar^{\tilde{N}} }{|\im(z)|},
	\end{aligned}
 \end{equation} 
 where the number $\tilde{N}$ is chosen sufficiently large. We will be choosing the parameter $\nu$ to be
  \begin{equation*}
  	\nu = \min\Big( 8 , \frac{|\im(z)|}{C_1 \hbar \log(\hbar^{-1})} \Big),
  \end{equation*}
where $C_1$ is a positive constant chosen sufficiently large to ensure the inevitability of the operator on the left-hand side of \eqref{EQ:localisation_propagator_7}. Note that $\nu$ will be of order one for all  $|\im(z)|\in[\mu_1,2\mu_1]$. With this choice for $\nu$ we get from \eqref{EQ:localisation_propagator_8} that
\begin{equation}\label{EQ:localisation_propagator_9}
	\begin{aligned}
	\big\lVert \OpW(\theta_2) (z-H_{\hbar,\mu,\varepsilon})^{-1} \OpW(\theta_1)\big\rVert_{\mathrm{op}} \leq{}& \frac{C}{|\im(z)|} \max\big( \hbar^{8}, e^{- \frac{|\im(z)|}{C_1 \hbar}} \big).
	 \end{aligned}
\end{equation}
Combining the estimates in \eqref{EQ:localisation_propagator_6} and \eqref{EQ:localisation_propagator_9} we obtain that
\begin{equation}\label{EQ:localisation_propagator_10}
	\big\lVert\OpW(\theta_2)(z-H_{\hbar,\mu})^{-1}\OpW(\theta_1) \big\rVert_{\mathrm{op}} \leq   \frac{C \max\big( \hbar^{8}, e^{- \frac{|\im(z)|}{C_1 \hbar}} \big)}{|\im(z)|}  + \frac{C\hbar^{5+\gamma+\frac{1}{2}(\kappa-\gamma')}}{|\im(z)|^2}.
\end{equation}
Using the estimate obtained in \eqref{EQ:localisation_propagator_10} we get that
\begin{equation}\label{EQ:localisation_propagator_11}
	\begin{aligned}
	\MoveEqLeft \frac{1}{\mu_1} \int_{\substack{-\eta <\re(z)
        <\eta \\ -2\mu_1\leq \im(z)\leq-\mu_1}} e^{- t \hbar^{-1} \im(z)} \big\lVert \OpW(\theta_2)(z-H_{\hbar,\mu})^{-1} \OpW(\theta_1) \big\rVert_{\mathrm{op}} \,L(dz)
       \\
       \leq{}&  \frac{C }{\mu_1^2}  \int_{\substack{-\eta <\re(z)
        <\eta \\ -2\mu_1\leq \im(z)\leq-\mu_1}} e^{2 T_1 \hbar^{-1} \mu_1} \big[  \max\big( \hbar^{8}, e^{- \frac{|\im(z)|}{C_1 \hbar}} \big)  + \hbar^{5+\gamma+\frac{1}{2}(\kappa-\gamma')} \mu_1^{-1} \big]\,L(dz)
        \\
       \leq{}&  \frac{C}{\mu_1}  e^{2 t \hbar^{-1} \mu_1}  \big( \hbar^{8}+e^{- \frac{\mu_1}{C_1 \hbar}} +  \hbar^{5+\gamma+\frac{1}{2}(\kappa-\gamma')} \mu_1^{-1} \big)  
       \\
       \leq{}&    \frac{C }{M}   \big( \hbar^{7 -2 T_1 M}+\hbar^{ M( \frac{1}{C_1} -2T_1) -1} + \hbar^{3+\gamma +\frac{1}{2}(\kappa-\gamma') - 2T_1 M} \big),      
        \end{aligned}
\end{equation}
where we in the above calculation have used that $\mu_1=M \hbar \log(\hbar^{-1})$. Recalling that $C_1$ is already a fixed constant we choose $M=C_1(4+\gamma +\frac{1}{2}(\kappa-\gamma'))$ and $T_1 = \frac{1}{4}(\kappa-\gamma') M^{-1}$. We then obtain that $2T_1 M = \frac{1}{2} (\kappa-\gamma') \leq \frac{1}{2}$. With these choices of constants we obtain from \eqref{EQ:localisation_propagator_11} that
\begin{equation}\label{EQ:localisation_propagator_12}
	\begin{aligned}
	\frac{1}{\mu_1} \int_{\substack{-\eta <\re(z)
        <\eta \\ -2\mu_1\leq \im(z)\leq-\mu_1}} e^{- t \hbar^{-1} \im(z)} \big\lVert \OpW(\theta_2)(z-H_{\hbar,\mu})^{-1} \OpW(\theta_1) \big\rVert_{\mathrm{op}} \,L(dz)
       \leq  C\hbar^{3+\gamma}.   
        \end{aligned}
\end{equation}
Analogously we get that
\begin{equation}\label{EQ:localisation_propagator_13}
	\begin{aligned}
	\MoveEqLeft \frac{1}{\mu_1} \int_{\substack{-\eta <\re(z)
        <\eta \\ \mu_1\leq \im(z)\leq2 \mu_1}}  \big\lVert \OpW(\theta_2) (z-H_{\hbar,\mu})^{-1}\OpW(\theta_1) \big\rVert_{\mathrm{op}} \,L(dz) 
        \leq  C\hbar^{3+\gamma}.  
        \end{aligned}
\end{equation}
Combining the estimates obtained in \eqref{EQ:localisation_propagator_3}, \eqref{EQ:localisation_propagator_5}, \eqref{EQ:localisation_propagator_12} and \eqref{EQ:localisation_propagator_13} we obtain that
\begin{equation*}
	\norm{\OpW(\theta_2) e^{it\hbar^{-1} H_{\hbar,\mu} } f(H_{\hbar,\mu}) \OpW(\theta_1)}_{\mathrm{op}} \leq C \hbar^{3+\gamma}.
\end{equation*}
This is the desired estimate for the case $0<t\leq T_1$. The case $-T_1\leq t<0$ is proven analogously. This concludes the proof.
\end{proof}

\begin{lemma}\label{LE:Func_moll_com_model_reg}
Let $\gamma\in[0,1]$ and $\mathcal{H}_{\hbar,\mu}$ be an operator acting in $L^2(\R^d)$. Suppose  $\mathcal{H}_{\hbar,\mu}$ and $\gamma$ satisfies Assumption~\ref{Assumption:local_potential_1} with the open set $\Omega$ and let   $H_{\hbar,\mu,\varepsilon} = (-i\hbar\nabla-\mu a)^2 +V_\varepsilon $ be the associated rough Schr\"odinger operator.  Assume that $\mu\leq\mu_0<1$ and $\hbar\in(0,\hbar_0]$, with $\hbar_0$ sufficiently small.
Moreover, let $\chi_\hbar(t)$ be the function from Remark~\ref{RE:mollyfier_def}, $f\in C_0^\infty(\R)$ and $\varphi\in C_0^\infty(\Omega)$ then it holds for $s\in\R$ that
\begin{equation} \label{EQLE:Func_moll_com_model_reg}
	\begin{aligned}
	\MoveEqLeft \norm{\varphi f(\mathcal{H}_{\hbar,\mu}) \chi_\hbar(\mathcal{H}_{\hbar,\mu}-s)f(\mathcal{H}_{\hbar,\mu})  \varphi -\varphi f(H_{\hbar,\mu,\varepsilon}) \chi_\hbar(H_{\hbar,\mu,\varepsilon}-s)f(H_{\hbar,\mu,\varepsilon})  \varphi }_1 
	\\
	&\leq C \varepsilon^{5+\kappa} \hbar^{-d-2} + C' \hbar^{1+\gamma-d}.
	\end{aligned}
\end{equation}
Moreover, suppose there exists some $c>0$ such that 
\begin{equation*}
	|V(x)| +\hbar^{\frac{2}{3}} \geq c \qquad\text{for all $x\in \Omega$}.
\end{equation*}
Then it holds that
\begin{equation}\label{LEEQ:Func_moll_com_model_reg}
	\norm{\varphi f(\mathcal{H}_{\hbar,\mu}) \chi_\hbar(\mathcal{H}_{\hbar,\mu}-s)f(\mathcal{H}_{\hbar,\mu})  \varphi }_1  \leq C \hbar^{-d}.
\end{equation}
The constants $C$ and $C'$ depend on $\supp(f)$, the dimension and the numbers $\norm{f}_{L^\infty(\R)}$, $\norm{\partial^\alpha\varphi}_{L^\infty(\R)}$ for all $\alpha\in\N^d_0$, $\norm{\partial^\alpha a_j}_{L^\infty(\R^d)}$ for all $\alpha\in\N_0^d$ and $j\in\{1\dots,d\}$ and $\norm{\partial_x^\alpha V}_{L^\infty(\R^d)}$ for all $\alpha\in\N_0^d$ such that $|\alpha|\leq 5$ and the H\"older constant for $V$.
\end{lemma}
\begin{proof}
Let $H_{\hbar,\mu}$ be the magenetic Schr\"odinger operator associated to $\mathcal{H}_{\hbar,\mu}$. We then have that
\begin{equation}\label{EQ:Func_moll_com_model_reg_0}
	\begin{aligned}
	\MoveEqLeft \norm{\varphi f(\mathcal{H}_{\hbar,\mu}) \chi_\hbar(\mathcal{H}_{\hbar,\mu}-s)f(\mathcal{H}_{\hbar,\mu})  \varphi -\varphi f(H_{\hbar,\mu,\varepsilon}) \chi_\hbar(H_{\hbar,\mu,\varepsilon}-s)f(H_{\hbar,\mu,\varepsilon})  \varphi }_1 
	\\
	\leq{}& \norm{\varphi f(\mathcal{H}_{\hbar,\mu}) \chi_\hbar(\mathcal{H}_{\hbar,\mu}-s)f(\mathcal{H}_{\hbar,\mu})  \varphi  -\varphi f(H_{\hbar,\mu}) \chi_\hbar(H_{\hbar,\mu}-s)f(H_{\hbar,\mu})  \varphi }_1 
	\\
	&+\norm{\varphi f(H_{\hbar,\mu}) \chi_\hbar(H_{\hbar,\mu}-s)f(H_{\hbar,\mu})  \varphi   - \varphi f(H_{\hbar,\mu,\varepsilon}) \chi_\hbar(H_{\hbar,\mu,\varepsilon}-s)f(H_{\hbar,\mu,\varepsilon})  \varphi }_1.
	\end{aligned}
\end{equation}
We start by estimating the first term on the righthand side of \eqref{EQ:Func_moll_com_model_reg_0}. By applying Lemma~\ref{LE:Comparision_Loc_infty}  we get that
\begin{equation}\label{EQ:Func_moll_com_model_reg_1}
	\begin{aligned}
	\MoveEqLeft\norm{\varphi f(\mathcal{H}_{\hbar,\mu}) \chi_\hbar(\mathcal{H}_{\hbar,\mu}-s)f(\mathcal{H}_{\hbar,\mu})  \varphi  -\varphi f(H_{\hbar,\mu}) \chi_\hbar(H_{\hbar,\mu}-s)f(H_{\hbar,\mu})  \varphi }_1 
	\\
	\leq{}& \norm{ \varphi f(\mathcal{H}_{\hbar,\mu}) \chi_\hbar(\mathcal{H}_{\hbar,\mu}-s) f(H_{\hbar,\mu})  \varphi  - \varphi f(\mathcal{H}_{\hbar,\mu}) \chi_\hbar(H_{\hbar,\mu}-s) f(H_{\hbar,\mu})\varphi }_{1}
	\\
	&+ C\hbar^{-1} \norm{\varphi f(\mathcal{H}_{\hbar,\mu})  - \varphi f(H_{\hbar,\mu})  }_1 + C\hbar^{-1} \norm{f(\mathcal{H}_{\hbar,\mu}) \varphi  - f(H_{\hbar,\mu}) \varphi   }_1
	\\
	\leq{}& \norm{\varphi f(\mathcal{H}_{\hbar,\mu}) \chi_\hbar(\mathcal{H}_{\hbar,\mu}-s) f(H_{\hbar,\mu})  \varphi  - \varphi f(\mathcal{H}_{\hbar,\mu})\chi_\hbar(H_{\hbar,\mu}-s) f(H_{\hbar,\mu})\varphi }_{1} + C \hbar^N,
	\end{aligned}
\end{equation}
where we in the first inequality have added and subtracted the two terms $\varphi f(\mathcal{H}_{\hbar,\mu}) \chi_\hbar(H_{\hbar,\mu}-s)f(H_{\hbar,\mu})  \varphi $ and $\varphi f(\mathcal{H}_{\hbar,\mu})\chi_\hbar(\mathcal{H}_{\hbar,\mu}-s) f(H_{\hbar,\mu})  \varphi $, used the triangle inequality, Lemma~\ref{LE:Comparision_Loc_infty} and that $\sup_{t\in\R}\chi_\hbar(t) \leq C\hbar^{-1}$. In the second inequality, we have used Lemma~\ref{LE:Func_com_model_reg}.  
We observe that, using how we defined the function $\chi_\hbar(z-s)$, we have that
\begin{equation*}
	\chi_\hbar(z-s) = \mathcal{F}_\hbar^{-1}[\chi] (z-s) = \frac{1}{2\pi \hbar} \int_\R e^{i\hbar^{-1}t(z-s)} \chi(t) \,dt.
\end{equation*}
Using this expression and the fundamental theorem of calculus we get that
\begin{equation}\label{EQ:Func_moll_com_model_reg_2}
	\begin{aligned}
	\MoveEqLeft \chi_\hbar(\mathcal{H}_{\hbar,\mu}-s) - \chi_\hbar(H_{\hbar,\mu}-s) 
	\\
	&= \frac{1}{2\pi \hbar} \int_\R \big(e^{i\hbar^{-1}t(\mathcal{H}_{\hbar,\mu}-s)} -e^{i\hbar^{-1}t(H_{\hbar,\mu}-s)}\big) \chi(t) \,dt
	\\
	&= \frac{i}{2\pi \hbar^2} \int_\R  e^{-i\hbar^{-1}t s} \chi(t)  \int_{0}^t e^{i\hbar^{-1}\tau\mathcal{H}_{\hbar,\mu}}\big( \mathcal{H}_{\hbar,\mu}- H_{\hbar,\mu} \big)e^{i\hbar^{-1}(t-\tau)H_{\hbar,\mu}} \,d\tau dt.
	\end{aligned}
\end{equation}
Letting $\tilde{f}\in C_0^\infty(\R)$ such that $\tilde{f}(t)f(t) = f(t)$ for all $t\in\R$ and using the identity obtained in \eqref{EQ:Func_moll_com_model_reg_2} we get that
\begin{equation}\label{EQ:Func_moll_com_model_reg_3}
	\begin{aligned}
	\MoveEqLeft \norm{\varphi f(\mathcal{H}_{\hbar,\mu}) \chi_\hbar(\mathcal{H}_{\hbar,\mu}-s) f(H_{\hbar,\mu})  \varphi  - \varphi f(\mathcal{H}_{\hbar,\mu})\chi_\hbar(H_{\hbar,\mu}-s) f(H_{\hbar,\mu})\varphi }_{1}
	\\
	\leq{}& \frac{1}{2\pi \hbar^2} \int_\R \chi(t)  \int_{0}^t  \big\lVert \varphi f(\mathcal{H}_{\hbar,\mu}) e^{i\hbar^{-1}t\mathcal{H}_{\hbar,\mu}}\tilde{f_1}(\mathcal{H}_{\hbar,\mu})\big( \mathcal{H}_{\hbar,\mu}- H_{\hbar,\mu} \big)
	\\
	&\phantom{ \frac{1}{2\pi \hbar^2} \int_\R \chi(t)  \int_{0}^t  \big\lVert \varphi}{}
	 \times \tilde{f}(H_{\hbar,\mu}) e^{i\hbar^{-1}(t-\tau)H_{\hbar,\mu}} f(H_{\hbar,\mu})\varphi \big\rVert_1 \,d\tau dt.
	\end{aligned}
\end{equation}
As in the proof of Lemma~\ref{LE:assump_est_func_loc} we let $\theta\in C_0^\infty(\Omega\times B(0,K+1))$ such that
\begin{equation*}
	\supp(\varphi)\cap \supp(1-\theta) \cap \supp(f(a_{\varepsilon,0}^f) )=\emptyset,
\end{equation*}
where $a_{\varepsilon,0}^f(x,p) = f( (p- \mu a(x))^2  + V_\varepsilon(x) )$.  Then by arguing as in the proof of Lemma~\ref{LE:assump_est_func_loc} and combining it with the estimate obtained in \eqref{EQ:Func_com_model_reg_4} we get that
\begin{equation}\label{EQ:Func_moll_com_model_reg_4}
	\begin{aligned}
	\MoveEqLeft \big\lVert \varphi f(\mathcal{H}_{\hbar,\mu}) e^{i\hbar^{-1}t\mathcal{H}_{\hbar,\mu}}\tilde{f}(\mathcal{H}_{\hbar,\mu})\big( \mathcal{H}_{\hbar,\mu}- H_{\hbar,\mu} \big) \tilde{f}(H_{\hbar,\mu}) e^{i\hbar^{-1}(t-\tau)H_{\hbar,\mu}} f(H_{\hbar,\mu})\varphi \big\rVert_1
	\\
	\leq{}& C\hbar^{-d}  \big\lVert \tilde{f}(\mathcal{H}_{\hbar,\mu})\big( \mathcal{H}_{\hbar,\mu}- H_{\hbar,\mu} \big) \tilde{f}(H_{\hbar,\mu}) e^{i\hbar^{-1}(t-\tau)H_{\hbar,\mu}} f(H_{\hbar,\mu}) \OpW(\theta) \varphi \big\rVert_{\mathrm{op}} + C \varepsilon^{5+\kappa} \hbar^{-d}.
	\end{aligned}
\end{equation}
Let $\tilde{\theta}\in C_0^\infty(\Omega\times B(0,K+1))$ such that 
\begin{equation*}
	\dist \big( \supp(\theta),  \supp(1-\tilde{\theta}) \big)\geq c,
\end{equation*}
where $c>0$ is some positive constant. With this function, we have that
\begin{equation}\label{EQ:Func_moll_com_model_reg_5}
	\begin{aligned}
	\MoveEqLeft  \big\lVert \tilde{f}(\mathcal{H}_{\hbar,\mu})\big( \mathcal{H}_{\hbar,\mu}- H_{\hbar,\mu} \big) \tilde{f}(H_{\hbar,\mu}) e^{i\hbar^{-1}(t-\tau)H_{\hbar,\mu}} f(H_{\hbar,\mu}) \OpW(\theta) \varphi \big\rVert_{\mathrm{op}} 
	\\
	\leq{}& C \big\lVert \tilde{f}(\mathcal{H}_{\hbar,\mu})\big( \mathcal{H}_{\hbar,\mu}- H_{\hbar,\mu} \big) \tilde{f}(H_{\hbar,\mu}) \OpW(\tilde{\theta})  \big\rVert_{\mathrm{op}} 
	\\
	&+ C\big\lVert \OpW(1-\tilde{\theta}) e^{i\hbar^{-1}(t-\tau)H_{\hbar,\mu}} f(H_{\hbar,\mu}) \OpW(\theta) \varphi \big\rVert_{\mathrm{op}}
	\\
	\leq{}& C \hbar^{3+\gamma},  
	\end{aligned}
\end{equation}
where we in the last inequality have used Lemma~\ref{LE:Func_com_model_reg} and Lemma~\ref{LE:localisation_propagator}. We have here also used that $t-\tau$ is sufficiently small due to the support properties of $\chi$. From combining the estimates in \eqref{EQ:Func_moll_com_model_reg_1}, \eqref{EQ:Func_moll_com_model_reg_3}, \eqref{EQ:Func_moll_com_model_reg_4} and \eqref{EQ:Func_moll_com_model_reg_5} we obtain that
\begin{equation}\label{EQ:Func_moll_com_model_reg_6}
	\begin{aligned}
	\MoveEqLeft \norm{\varphi f(\mathcal{H}_{\hbar,\mu}) \chi_\hbar(\mathcal{H}_{\hbar,\mu}-s)f(\mathcal{H}_{\hbar,\mu})  \varphi  -\varphi f(H_{\hbar,\mu}) \chi_\hbar(H_{\hbar,\mu}-s)f(H_{\hbar,\mu})  \varphi }_1 
	\\
	&\leq C\big( \hbar^{1+\gamma-d} +  \varepsilon^{5+\kappa} \hbar^{-2-d}\big).
	\end{aligned}
\end{equation}
For the second term on the righthand side of \eqref{EQ:Func_moll_com_model_reg_0} we get by a triangle inequality and the estimate obtained in \eqref{EQ:Func_com_model_reg_4} that 
\begin{equation}\label{EQ:Func_moll_com_model_reg_7}
	\begin{aligned}
	\MoveEqLeft \norm{\varphi f(H_{\hbar,\mu}) \chi_\hbar(H_{\hbar,\mu}-s)f(H_{\hbar,\mu})  \varphi   - \varphi f(H_{\hbar,\mu,\varepsilon}) \chi_\hbar(H_{\hbar,\mu,\varepsilon}-s)f(H_{\hbar,\mu,\varepsilon})  \varphi }_1
	\\
	&\leq \big\lVert \varphi f(H_{\hbar,\mu}) \big[ \chi_\hbar(H_{\hbar,\mu}-s)  - \chi_\hbar(H_{\hbar,\mu,\varepsilon}-s) \big]f(H_{\hbar,\mu})  \varphi \big\rVert_1+  C \varepsilon^{5+\kappa} \hbar^{-1-d}
	\\
	&\leq C\hbar^{-d} \big\lVert   \chi_\hbar(H_{\hbar,\mu}-s)  - \chi_\hbar(H_{\hbar,\mu,\varepsilon}-s) \big\rVert_{\mathrm{op}}+  C \varepsilon^{5+\kappa} \hbar^{-1-d},
	\end{aligned}
\end{equation}
where we in the last inequality have used Lemma~\ref{LE:Func_cal_est_inf_pon}. Using the same fundamental theorem of calculus as in \eqref{EQ:Func_moll_com_model_reg_2} we get that
\begin{equation}\label{EQ:Func_moll_com_model_reg_8}
	\begin{aligned}
	\MoveEqLeft \big\lVert   \chi_\hbar(H_{\hbar,\mu}-s)  - \chi_\hbar(H_{\hbar,\mu,\varepsilon}-s) \big\rVert_{\mathrm{op}}
	\\
	&\leq  \frac{1}{2\pi \hbar^2} \int_\R   \chi(t)  \int_{0}^t \big\lVert H_{\hbar,\mu} - H_{\hbar,\mu,\varepsilon}  \big\rVert_{\mathrm{op}} \,d\tau dt
	\leq C \varepsilon^{5+\kappa} \hbar^{-2}.
	\end{aligned}
\end{equation}
From combining the estimates in \eqref{EQ:Func_moll_com_model_reg_7} and \eqref{EQ:Func_moll_com_model_reg_8} we obtain that
\begin{equation}\label{EQ:Func_moll_com_model_reg_9}
	\begin{aligned}
	\norm{\varphi f(H_{\hbar,\mu}) \chi_\hbar(H_{\hbar,\mu}-s)f(H_{\hbar,\mu})  \varphi   - \varphi f(H_{\hbar,\mu,\varepsilon}) \chi_\hbar(H_{\hbar,\mu,\varepsilon}-s)f(H_{\hbar,\mu,\varepsilon})  \varphi }_1
	\leq C \varepsilon^{5+\kappa} \hbar^{-2-d}.
	\end{aligned}
\end{equation}
Finally by combining the estimates in \eqref{EQ:Func_moll_com_model_reg_0}, \eqref{EQ:Func_moll_com_model_reg_7} and \eqref{EQ:Func_moll_com_model_reg_9} we obtain the estimate stated in \eqref{EQLE:Func_moll_com_model_reg}. By combining the estimate in \eqref{EQLE:Func_moll_com_model_reg} with Lemma~\ref{LE:assump_est_func_loc} we can obtain the estimate \eqref{LEEQ:Func_moll_com_model_reg}. This concludes the proof.
\end{proof}
\begin{lemma}\label{LE:trace_com_model_reg}
Let $\mathcal{H}_{\hbar,\mu}$ be an operator acting in $L^2(\R^d)$. Suppose  $\mathcal{H}_{\hbar,\mu}$ satisfies Assumption~\ref{Assumption:local_potential_1} with the open set $\Omega$ and let   $H_{\hbar,\mu,\varepsilon} = (-i\hbar\nabla-\mu a)^2 +V_\varepsilon $ be the associated rough Schr\"odinger operator.  Assume that $\mu\leq\mu_0<1$ and $\hbar\in(0,\hbar_0]$, with $\hbar_0$ sufficiently small.
Moreover, suppose there exists some $c>0$ such that 
\begin{equation*}
	|V(x)| +\hbar^{\frac{2}{3}} \geq c \qquad\text{for all $x\in \Omega$},
\end{equation*}
and let $\varphi\in C_0^\infty(B(0,R))$. Then for $g\in C^{\infty,\gamma}(\R)$ with $\gamma\in[0,1]$ it holds that
\begin{equation}
	\Big|\Tr[\varphi g(\mathcal{H}_{\hbar,\mu})] -\Tr[\varphi g(H_{\hbar,\mu,\varepsilon})] \Big|  \leq C \hbar^{1+\gamma-d} + C' \varepsilon^{5+\kappa} \hbar^{-d-2}.
\end{equation}
The constants $C$ and $C'$ depends on $\supp(f)$, the dimension and the numbers $\norm{f}_{L^\infty(\R)}$, $\norm{\partial^\alpha\varphi}_{L^\infty(\R)}$ for all $\alpha\in\N^d_0$, $\norm{\partial^\alpha a_j}_{L^\infty(\R^d)}$ for all $\alpha\in\N_0^d$ with $|\alpha|\geq1$ and $j\in\{1\dots,d\}$ and $\norm{\partial_x^\alpha V}_{L^\infty(\R^d)}$ for all $\alpha\in\N_0^d$ such that $|\alpha|\leq 5$ and the H\"older constant for $V$.
\end{lemma}
\begin{proof}
All quantities are Gauge invariant. Hence we can start by performing a Gauge transform such that the supremum norm of the magnetic vector potential is uniformly bounded.

Since both operators are lower semi-bounded we may assume that $g$ is compactly supported. Let $f\in C_0^\infty(\R)$ such that $f(t)g(t)= g(t)$ for all $t\in\R$. Let $\varphi_1\in C_0^\infty(\Omega)$ such that $\varphi(x)\varphi_1(x) = \varphi(x)$ for all $x\in\R^d$. Moreover, let $\chi_\hbar(t)$ be the function from Remark~\ref{RE:mollyfier_def} and  set $g^{(\hbar)}(t) = g*\chi_\hbar(t)$. With this notation set up, we have that
\begin{equation}\label{EQ:trace_com_model_reg_1}
	\begin{aligned}
	\MoveEqLeft \Big|\Tr[\varphi g(\mathcal{H}_{\hbar,\mu})] -\Tr[\varphi g(H_{\hbar,\mu,\varepsilon})] \Big|
	\\
	\leq {}& \lVert \varphi \varphi_1f(\mathcal{H}_{\hbar,\mu}) (g(\mathcal{H}_{\hbar,\mu}) - g^{(\hbar)}(\mathcal{H}_{\hbar,\mu}))f(\mathcal{H}_{\hbar,\mu})\varphi_1 \rVert_1
	\\
	&+\lVert \varphi \varphi_1 f(H_{\hbar,\mu,\varepsilon}) (g(H_{\hbar,\mu,\varepsilon})-g^{(\hbar)}(H_{\hbar,\mu,\varepsilon}))f(H_{\hbar,\mu,\varepsilon})\varphi_1 \rVert_1
	+ \norm{\varphi}_{L^\infty(\R^d)} \int_{\R} g_\gamma(s)  \,ds
	\\
	&\times  \sup_{s\in\R} \lVert\varphi \varphi_1 f(\mathcal{H}_{\hbar,\mu}) \chi_\hbar(\mathcal{H}_{\hbar,\mu}-s)f(\mathcal{H}_{\hbar,\mu})  \varphi_1 -\varphi_1 f(H_{\hbar,\mu,\varepsilon}) \chi_\hbar(H_{\hbar,\mu,\varepsilon}-s)f(H_{\hbar,\mu,\varepsilon})  \varphi_1\rVert_1.
	\end{aligned}
\end{equation}
Lemma~\ref{LE:assump_est_func_loc} and Lemma~\ref{LE:Func_moll_com_model_reg} gives us that the assumptions of Proposition~\ref{PRO:Tauberian} is fulfilled with $B$ equal to $\varphi_1f(\mathcal{H}_{\hbar,\mu})$ and $\varphi_1 f(H_{\hbar,\mu,\varepsilon})$  respectively. Hence we have that 
\begin{equation}\label{EQ:trace_com_model_reg_2}
	\begin{aligned}
	 \lVert \varphi \varphi_1f(\mathcal{H}_{\hbar,\mu}) (g(\mathcal{H}_{\hbar,\mu}) - g^{(\hbar)}(\mathcal{H}_{\hbar,\mu}))f(\mathcal{H}_{\hbar,\mu})\varphi_1 \rVert_1
	\leq C \hbar^{1+\gamma-d}
	\end{aligned}
\end{equation}
and
\begin{equation}\label{EQ:trace_com_model_reg_3}
	\begin{aligned}
	\lVert \varphi \varphi_1 f(H_{\hbar,\mu,\varepsilon}) (g(H_{\hbar,\mu,\varepsilon})-g^{(\hbar)}(H_{\hbar,\mu,\varepsilon}))f(H_{\hbar,\mu,\varepsilon})\varphi_1] \rVert_1 \leq C \hbar^{1+\gamma-d}.
	\end{aligned}
\end{equation}
From applying Lemma~\ref{LE:Func_moll_com_model_reg} we get that
\begin{equation}\label{EQ:trace_com_model_reg_4}
	\begin{aligned}
	\MoveEqLeft  \sup_{s\in\R} \lVert\varphi \varphi_1 f(\mathcal{H}_{\hbar,\mu}) \chi_\hbar(\mathcal{H}_{\hbar,\mu}-s)f(\mathcal{H}_{\hbar,\mu})  \varphi_1 -\varphi_1 f(H_{\hbar,\mu,\varepsilon}) \chi_\hbar(H_{\hbar,\mu,\varepsilon}-s)f(H_{\hbar,\mu,\varepsilon})  \varphi_1\rVert_1
	 \\
	 &\leq C\varepsilon^{5+\kappa} \hbar^{-d-2} + C'\hbar^{1+\gamma-d}.
	\end{aligned}
\end{equation}
Finally from combining the estimates in \eqref{EQ:trace_com_model_reg_1}, \eqref{EQ:trace_com_model_reg_2}, \eqref{EQ:trace_com_model_reg_3} and \eqref{EQ:trace_com_model_reg_4}we obtain the desired estimate and this concludes the proof.
\end{proof}
\section{Local model problem}\label{sec:model_prob}
Before we state and prove our local model problem we will state a result on comparison of phase-space integrals that we will need later.
\begin{lemma}\label{LE:comparison_phase_space_int}
Suppose $\Omega\subset\R^d$ is an open set, $d\geq2$ and let $\varphi\in C_0^\infty(\Omega)$. Moreover, let $\varepsilon>0$, $\hbar\in(0,\hbar_0]$ and  $V,V_\varepsilon\in L^1_{loc}(\R^d)\cap C(\Omega)$. Suppose for some number $\tau\geq0$ that 
\begin{equation}\label{EQLE:comparison_phase_space_int}
	\norm{V-V_\varepsilon}_{L^\infty(\Omega)}\leq c\varepsilon^{\tau}.
\end{equation}
Then for $\gamma\in[0,1]$ and $\varepsilon$ sufficiently small it holds that
\begin{equation}\label{LEEQ:Loc_mod_prob_5}
	\begin{aligned}
	 \Big| \int_{\R^{2d}} [g_\gamma(p^2+V_\varepsilon(x))-g_\gamma(p^2+V(x))]\varphi(x) \,dx dp  \Big| 
	 \leq C\varepsilon^{\tau},
	  \end{aligned}
\end{equation}
where the constant $C$ depends on the dimension and the numbers $\gamma$ and $c$  in  \eqref{EQLE:comparison_phase_space_int}.
\end{lemma}
\begin{proof}
Firstly we observe that due to \eqref{EQLE:comparison_phase_space_int}  we have that 
\begin{equation}\label{EQ:comparison_phase_space_int_1}
	\sup_{x\in\Omega}\big|V(x)_{-}-V_\varepsilon(x)_{-}\big|\leq c\varepsilon^{\tau}.
\end{equation}
To compare the phase-space integrals we start by evaluating the integral in $p$. This yields
\begin{equation}\label{EQ:comparison_phase_space_int_2}
	\begin{aligned}
	\MoveEqLeft \int_{\R^{2d}} [g_\gamma(p^2+V_\varepsilon(x))-g_\gamma(p^2+V(x))]\varphi(x) \,dx dp 
	\\
	&= L_{\gamma,d}^{\mathrm{cl}} \int_{\R^d} \big[V_\varepsilon(x)_{-}^{\frac{d}{2}+\gamma} -  V(x)_{-}^{\frac{d}{2}+\gamma}\big] \varphi(x) \,dx,
	  \end{aligned}
\end{equation}
where the constant $L_{\gamma,d}^{\mathrm{cl}} $ is given by
\begin{equation*}
	\begin{aligned}
	L_{\gamma,d}^{\mathrm{cl}} = \frac{\Gamma(\gamma+1)}{(4\pi)^{\frac{d}{2}}\Gamma(\gamma+\frac{d}{2}+1)},
	\end{aligned}
\end{equation*}
where $\Gamma$ is the standard gamma function. Since both $V_\varepsilon(x)_{-}$ and $V(x)_{-}$ are bounded from below and $d\geq2$ we can use that the map $r\mapsto r^{\frac{d}{2}+\gamma}$ is uniformly Lipschitz continuous when restricted to a compact domain. This gives us that
\begin{equation}\label{EQ:comparison_phase_space_int_3}
	\begin{aligned}
	 \int_{\R^d} \big[V_\varepsilon(x)_{-}^{\frac{d}{2}+\gamma} -  V(x)_{-}^{\frac{d}{2}+\gamma}\big] \varphi(x) \,dx \leq C_\gamma \int_{\R^d} \big|V_\varepsilon(x)_{-} -  V(x)_{-}\big| \varphi(x) \,dx
	 \leq \tilde{C}_\gamma \varepsilon^{\tau},
	  \end{aligned}
\end{equation}
where we have used \eqref{EQ:comparison_phase_space_int_1} and that $\supp(\varphi)\subset \Omega$. From combining \eqref{EQ:comparison_phase_space_int_2} and \eqref{EQ:comparison_phase_space_int_3} we obtain the desired estimate and this concludes the proof.
\end{proof}
With this established, we can now state our model problem.
\begin{thm}\label{THM:Loc_mod_prob}
Let $\mathcal{H}_{\hbar,\mu}$ be an operator acting in $L^2(\R^d)$ with $d\geq2$ and let  $\gamma\in[0,1]$. Suppose  $\mathcal{H}_{\hbar,\mu}$ and $\gamma$ satisfies Assumption~\ref{Assumption:local_potential_1} with the open set $\Omega$ and let   $H_{\hbar,\mu} = (-i\hbar\nabla-\mu a)^2 +V$ be the associated Schr\"odinger operator. Assume that $\mu\leq\mu_0<1$ and $\hbar\in(0,\hbar_0]$, with $\hbar_0$ sufficiently small. Moreover, suppose there exists some $c>0$ such that 
\begin{equation*}
	|V(x)| +\hbar^{\frac{2}{3}} \geq c \qquad\text{for all $x\in\Omega$}.
\end{equation*}
Then for any $\varphi\in C_0^\infty(\Omega)$ it holds that
\begin{equation*}
	  \Big|\Tr[\varphi g_\gamma(\mathcal{H}_{\hbar,\mu})] - \frac{1}{(2\pi\hbar)^d} \int_{\R^{2d}}g_\gamma(p^2+V(x)) \varphi(x) \,dx dp \Big| \leq C \hbar^{1+\gamma-d},
\end{equation*}
where the constant $C$ is depending on the dimension, the numbers $\norm{\partial^\alpha\varphi}_{L^\infty(\R)}$ for all $\alpha\in\N^d_0$, $\norm{\partial^\alpha a_j}_{L^\infty(\R^d)}$ for all $\alpha\in\N_0^d$ with $|\alpha|\geq1$ and $j\in\{1\dots,d\}$ and $\norm{\partial_x^\alpha V}_{L^\infty(\R^d)}$ for all $\alpha\in\N_0^d$ such that $|\alpha|\leq 5$ and the H\"older constant for $V$.
\end{thm}
\begin{proof}[Proof of Theorem~\ref{THM:Loc_mod_prob}]
Let $H_{\hbar,\mu,\varepsilon} = (-i\hbar\nabla-\mu a)^2 +V_\varepsilon $ be the rough Schr\"odinger operator associated to $\mathcal{H}_{\hbar,\mu}$. We have in the construction of $V_\varepsilon$ chosen $\varepsilon=\hbar^{1-\delta}$, where
\begin{equation}
	\delta=\frac{\kappa-\gamma}{5+\kappa}.
\end{equation}
Note that since we assume $2\geq \kappa>\gamma$, we have that $1>\delta>0$. With this choice of $\varepsilon$ and $\delta$ we have that
\begin{equation}\label{EQ:Loc_mod_prob_1}
	\varepsilon^{5+\kappa} = \hbar^{(1-\delta)(5+\kappa)} = \hbar^{5+\gamma}. 
\end{equation}
Moreover, since we have assumed a non-critical condition for  our original problem we get  that there exists a constant $\tilde{c}$ such that for all $\varepsilon$ sufficiently small it holds that
\begin{equation*}
	|V_\varepsilon(x)| +\hbar^{\frac{2}{3}} \geq \tilde{c} \qquad\text{for all $x\in\Omega$}.
\end{equation*}
With this in place, we have that
\begin{equation}\label{EQ:Loc_mod_prob_2}
	\begin{aligned}
	  \MoveEqLeft \Big|\Tr[\varphi g_\gamma(\mathcal{H}_{\hbar,\mu})] - \frac{1}{(2\pi\hbar)^d} \int_{\R^{2d}}g_\gamma(p^2+V(x)) \varphi(x) \,dx dp \Big| 
	  \\
	  \leq{}&  \Big|\Tr[\varphi g_\gamma(\mathcal{H}_{\hbar,\mu}) -\Tr[\varphi g_\gamma(H_{\hbar,\mu,\varepsilon})] \Big|
	  \\
	  &+  \Big|\Tr[\varphi g_\gamma(H_{\hbar,\mu,\varepsilon})] - \frac{1}{(2\pi\hbar)^d} \int_{\R^{2d}}g_\gamma(p^2+V_\varepsilon(x)) \varphi(x) \,dx dp \Big|
	  \\
	  &+  \Big|\frac{1}{(2\pi\hbar)^d} \int_{\R^{2d}} (g_\gamma(p^2+V_\varepsilon(x))- g_\gamma(p^2+V(x))) \varphi(x) \,dx dp  \Big|.
	  \end{aligned}
\end{equation}
We have by Lemma~\ref{LE:trace_com_model_reg} that
\begin{equation}\label{EQ:Loc_mod_prob_3}
	 \Big|\Tr[\varphi g_\gamma(\mathcal{H}_{\hbar,\mu}) -\Tr[\varphi g_\gamma(H_{\hbar,\mu,\varepsilon})] \Big|  \leq C \hbar^{1+\gamma-d} + C \varepsilon^{5+\kappa} \hbar^{-d-2} \leq \tilde{C} \hbar^{1+\gamma-d},
\end{equation}
where we in the last equality have used \eqref{EQ:Loc_mod_prob_1}. From Lemma~\ref{LE:asymp_est_func_loc} we get that
\begin{equation}\label{EQ:Loc_mod_prob_4}
	\begin{aligned}
	 \Big|\Tr[\varphi g_\gamma(H_{\hbar,\mu,\varepsilon})] - \frac{1}{(2\pi\hbar)^d} \int_{\R^{2d}}g_\gamma(p^2+V_\varepsilon(x)) \varphi(x) \,dx dp \Big| \leq C\hbar^{1+\gamma-d}.
	  \end{aligned}
\end{equation}
To estimate the last contribution in \eqref{EQ:Loc_mod_prob_2} we first notice that by the construction of $V_\varepsilon$ we have that 
\begin{equation*}
	\sup_{x\in\Omega}\big|V(x)_{-}-V_\varepsilon(x)_{-}\big|\leq C\varepsilon^{5+\kappa} = C \hbar^{5+\gamma}.
\end{equation*}
Hence it follow from Lemma~\ref{LE:comparison_phase_space_int} that 
\begin{equation}\label{EQ:Loc_mod_prob_9}
	\begin{aligned}
	  \Big|\frac{1}{(2\pi\hbar)^d} \int_{\R^{2d}} (g_\gamma(p^2+V_\varepsilon(x))- g_\gamma(p^2+V(x))) \varphi(x) \,dx dp  \Big| \leq C\hbar^{5+\gamma-d}.
	  \end{aligned}
\end{equation}
Finally by combining \eqref{EQ:Loc_mod_prob_2}, \eqref{EQ:Loc_mod_prob_3}, \eqref{EQ:Loc_mod_prob_4} and \eqref{EQ:Loc_mod_prob_9} we obtain the desired estimate and this concludes the proof.
\end{proof}

\section{Proof of Theorem~\ref{THM:Local_five_derivative}}\label{sec:proof_main}
This section is devoted to the proof of Theorem~\ref{THM:Local_five_derivative}. The proof is based on the multi-scale techniques of
\cite{MR1343781} (see also \cite{MR1631419,MR1240575}). Before we start the proof we will recall the following Lemma from  \cite{MR1343781} where it is Lemma 5.4.
\begin{lemma}\label{LE:partition_lemma}
  Let $\Omega\subset\R^d$ be an open set and let $l$ be a function in
  $C^1(\bar{\Omega})$ such that $l>0$ on $\bar{\Omega}$ and assume
  that there exists $\rho$ in $(0,1)$ such that
  \begin{equation}
    \abs{\nabla_x l(x)} \leq \rho,
  \end{equation}
  for all $x$ in $\Omega$. 
  
  Then
  \begin{enumerate}[label=$\roman*)$]
	
  \item There exists a sequence $\{x_k\}_{k=0}^\infty$ in $\Omega$
    such that the open balls $B(x_k,l(x_k))$ form a covering of
    $\Omega$. Furthermore, there exists a constant $N_\rho$, depending only
    on the constant $\rho$, such that the intersection of more than
    $N_\rho$ balls are empty.
		
  \item One can choose a sequence $\{\varphi_k\}_{k=0}^\infty$ such
    that $\varphi_k \in C_0^\infty(B(x_k,l(x_k)))$ for all $k$ in
    $\N$. Moreover, for all multiindices $\alpha$ and all $k$ in $\N$
    \begin{equation*}
      \abs{\partial_x^\alpha \varphi_k(x)}\leq C_\alpha l(x_k)^{-{\abs{\alpha}}},
    \end{equation*} 	   
    and
    \begin{equation*}
      \sum_{k=1}^\infty \varphi_k(x) = 1,
    \end{equation*}
    for all $x$ in $\Omega$.
  \end{enumerate}
\end{lemma}
The proof of the Lemma is analogous to the proof of \cite[Theorem 1.4.10]{MR1996773}.
Before we give a proof of Theorem~\ref{THM:Local_five_derivative} we will prove the following theorem, where we have an additional assumption on the magnetic field compared to Theorem~\ref{THM:Local_five_derivative}.
\begin{thm}\label{THM:Local_five_derivative_almost_there}
Let $\mathcal{H}_{\hbar,\mu}$ be an operator acting in $L^2(\R^d)$ and let $\gamma\in[0,1]$. If $\gamma=0$ we assume $d\geq3$ and if $\gamma\in(0,1]$ we assume $d\geq4$. Suppose that $\mathcal{H}_{\hbar,\mu}$ and $\gamma$ satisfy Assumption~\ref{Assumption:local_potential_1} with the set $\Omega$ and the functions $V$ and $a_j$ for $j\in\{1,\dots,d\}$. 
Then for any $\varphi\in C_0^\infty(\Omega)$ it holds that
\begin{equation*}
	  \Big|\Tr[\varphi g_\gamma(\mathcal{H}_{\hbar,\mu})] - \frac{1}{(2\pi\hbar)^d} \int_{\R^{2d}}g_\gamma(p^2+V(x)) \varphi(x) \,dx dp \Big| \leq C \hbar^{1+\gamma-d}
\end{equation*}
for all $\hbar\in(0,\hbar_0]$ and $\mu\leq \mu_0<1$, $\hbar_0$ sufficiently small. The constant $C$ is depending on the dimension, the numbers $\norm{\partial^\alpha \varphi}_{L^\infty(\R^d)}$ and  $\norm{\partial^\alpha a_j}_{L^\infty(\R^d)}$ for all $\alpha\in\N_0^d$ and $j\in\{1\dots,d\}$ and $\norm{\partial_x^\alpha V}_{L^\infty(\R^d)}$ for all $\alpha\in\N_0^d$ such that $|\alpha|\leq 5$ and the H\"older constant for $V$.
\end{thm}
\begin{proof}
Since $\varphi\in C_0^\infty(\Omega)$  there exists a number $\epsilon>0$ such that
\begin{equation*}
	\dist(\supp(\varphi), \Omega^{c}) >\epsilon.
\end{equation*}
We need this number to ensure we stay in the region where $\mathcal{H}_{
\hbar,\mu}$ behaves as a magnetic Schr\"odinger operator. We let 
\begin{equation*}
	l(x) = A^{-1}\sqrt{ |V(x)|^2 + \hbar^\frac{4}{3}} \quad\text{and}\quad f(x)=\sqrt{l(x)}. 
\end{equation*} 
Where we choose $A >0$ sufficiently large such that
\begin{equation}\label{EQ:Global_five_derivative_1}
	l(x) \leq  \min\big(\tfrac{\epsilon}{11},1\big)  \quad\text{and}\quad |\nabla l(x)|\leq \rho <\frac{1}{8}
\end{equation} 
for all $x\in\overline{\supp(\varphi)}$. Note that we can choose $A$ independent of $\hbar$ and uniformly for $\hbar\in(0,\hbar_0]$. Moreover, we have that
\begin{equation}\label{EQ:Global_five_derivative_2}
	|V(x)| \leq A l(x). 
\end{equation}
We use Lemma~\ref{LE:partition_lemma} with the set $\supp(\varphi)$ and the function $l(x)$. We can do this since we have that $l>0$ due to the presence of $\hbar$ in the definition of $l$. By Lemma~\ref{LE:partition_lemma} with the set $\supp(\varphi)$ and the function $l(x)$ there exists a sequence $\{x_k\}_{k=1}^\infty$ in $\supp(\varphi)$ such that $\supp(\varphi) \subset \cup_{k\in\N} B(x_k,l(x_k))$ and there exists a constant $N_{\frac{1}{8}}$ such that at most $N_{\frac{1}{8}}$ of the sets $B(x_k,l(x_k))$ can have a non-empty intersection. Moreover there exists a sequence $\{\varphi_{k}\}_{k=1}^\infty$ such that $\varphi_k\in C_0^\infty(B(x_k,l(x_k))$,
\begin{equation}\label{EQ:Global_five_derivative_2.5}
	\big| \partial_x^\alpha \varphi_k(x) \big| \leq C_\alpha l(x_k)^{-|\alpha|} \qquad\text{for all $\alpha\in\N_0$},
\end{equation}
and
\begin{equation*}
	\sum_{k=1}^\infty \varphi_k(x) =1  \qquad\text{for all $x\in\supp(\varphi)$}.
\end{equation*}
We have that $ \cup_{k\in\N} B(x_k,l(x_k))$ is an open covering of $\supp(\varphi)$ and since this set is compact there exists a finite subset $\mathcal{I}'\subset \N$ such that  
\begin{equation*}
	\supp(\varphi) \subset \bigcup_{k\in\mathcal{I}'} B(x_k,l(x_k)).
\end{equation*}
To ensure that we have a finite partition of unity over the set $\supp(\varphi)$ we define the set
\begin{equation*}
	\mathcal{I} = \bigcup_{j\in\mathcal I'} \Big\{ k\in\N \,\big|\, B(x_k,l(x_k))\cap B(x_j,l(x_j)) \neq \emptyset \Big\}.
\end{equation*}
We have that $\mathcal{I} $ is still finite since at most $N_{\frac{1}{8}}$ balls can have a non-empty intersection. Moreover, we have that
\begin{equation*}
	\sum_{k\in\mathcal{I}} \varphi_k(x) =1  \qquad\text{for all $x\in\supp(\varphi)$}.
\end{equation*}
From this, we get the following identity
\begin{equation}\label{EQ:Rough_weyl_asymptotics_2.6}
	 \Tr[\varphi \boldsymbol{1}_{ (-\infty,0]}(H_{\hbar,\varepsilon})] = \sum_{k\in\mathcal{I}}  \Tr[\varphi_k \varphi \boldsymbol{1}_{ (-\infty,0]}(H_{\hbar,\varepsilon})] ,
\end{equation}
where we have used the linearity of the trace. In what follows we will use the following notation 
\begin{equation*}
	l_k=l(x_k), \quad f_k=f(x_k), \quad h_k = \frac{\hbar}{l_kf_k} \quad\text{and}\quad \mu_k = \frac{\mu l_k}{f_k}.
\end{equation*}
We have that $h_k$ is uniformly bounded from above since
\begin{equation*}
	l(x)f(x) = A^{-\frac32}(\abs{V_\varepsilon(x)}^2+\hbar^{\frac43})^{\frac34} \geq A^{-\frac32} \hbar, 
\end{equation*}
for all $x\in\R^d$. Moreover, due to our choice of $f$ and $l$ we have that $\mu_k$ is bounded from above by $\mu_0$ since for all $x\in\R^d$ we have that
\begin{equation}\label{EQ:Rough_weyl_asymptotics_3}
	\frac{ l(x)}{f(x)}  = \sqrt{l(x)} \leq 1.
\end{equation}
 We define the two unitary operators $U_l$ and $T_z$ by
\begin{equation*}
	U_l f(x) = l^{\frac{d}{2}} f( l x) \quad\text{and}\quad T_zf(x)=f(x+z) \qquad\text{for $f\in L^2(\R^d)$}.
\end{equation*}
Moreover, we set
\begin{equation*}
	\begin{aligned}
	\tilde{\mathcal{H}}_{\hbar_k,\mu_k} = f_k^{-2} (T_{x_k} U_{l_k}) \mathcal{H}_{\hbar,\mu} (T_{x_k} U_{l_k})^{*}.
	\end{aligned}
\end{equation*}
Since we have that $\mathcal{H}_{\hbar,\mu}$ satisfies Assumption~\ref{Assumption:local_potential_1} with the open set $\Omega$ and the functions $V$ and $a_j$ for all $j\in\{1,\dots,d\}$. We have that $\tilde{\mathcal{H}}_{h_k,\mu_k}$ will satisfies Assumption~\ref{Assumption:local_potential_1} with the open set $B(0,10)$ and the functions $\tilde{V}_k$ and $\tilde{a}_{l,k}$ for all $j\in\{1,\dots,d\}$, where
\begin{equation}
	\tilde{V}_k(x)=f_k^{-2} V(l_kx+x_k) \quad\text{and}\quad  \tilde{a}_{l,k}(x) =  l_k^{-1} a_j(l_kx+x_k) \quad\text{for all $j\in\{1,\dots,d\}$}.
\end{equation}
We will here need to establish that this rescaled operator satisfies the assumptions of Theorem~\ref{THM:Loc_mod_prob} with the parameters $h_k$ and $\mu_k$ and the set $B(0,8)$. Since we have that $h_k$ is bounded from above and $\mu_k\leq\mu_0$ as established above what remains is to verify that we have a non-critical condition.
To establish this we firstly observe that  by \eqref{EQ:Global_five_derivative_1} we have that
\begin{equation}\label{EQ:Global_five_derivative_3.5}
	(1-8\rho) l_k \leq l(x) \leq (1+8\rho) l_k \qquad\text{for all $x \in B(x_k,8l_k)$}.
\end{equation}
Using \eqref{EQ:Global_five_derivative_3.5} we have for for $x$ in $B(0,8)$ that
\begin{equation*}
	\begin{aligned}
	\abs{\tilde{V}_k(x)} + h_k^{\frac{2}{3}} &= f_k^{-2} \abs{V(l_kx+x_k)} + (\tfrac{\hbar}{f_k l_k})^{\frac{2}{3}}
	=l_k^{-1}( \abs{V(l_kx+x_k)} +\hbar^{\frac23})
	\\
	&\geq  l_k^{-1} A l(l_k x+x_k) \geq (1-8\rho) A.
	\end{aligned}
\end{equation*}
Hence we have a non-critical assumption for all $x\in B(0,8)$. What remains is to verify that the norms of the functions $\widetilde{\varphi_k\varphi}=(T_{x_k} U_{l_k})\varphi_k\varphi(T_{x_k} U_{l_k})^{*}$, $\tilde{V}_k$ and $\tilde{a}_{l,k}$ for all $j\in\{1,\dots,d\}$ are independent of $\hbar$ and $k$.
Due to  \eqref{EQ:Global_five_derivative_2} and that $l$ is slowly varying \eqref{EQ:Global_five_derivative_3.5} we have that
\begin{equation*}
	\norm{\tilde{V}_k}_{L^\infty(B(0,8))} = \sup_{x\in B(0,8)} \big| f_k^{-2} V(l_kx+x_k) \big| \leq A.
\end{equation*}
For $\alpha\in\N^d_0$ with $1\leq|\alpha|\leq5$ we have that
\begin{equation*}
	\begin{aligned}
    	 \norm{\partial_x^\alpha \tilde{V}_k(x) }_{L^\infty(B(0,8))} = f_k^{-2} l_k^{|\alpha|} \sup_{x\in B(0,8)}  |(\partial_x^\alpha V)(l_kx+x_k)|  \leq   \norm{\partial_x^\alpha V(x) }_{L^\infty(\R^d)}.
	\end{aligned}
\end{equation*}
Let $C_V$ be the H\"older constant for $V$. We then have for $\alpha\in\N^d_0$ with $|\alpha|=5$ that
\begin{equation*}
	\begin{aligned}
    	 \big| \partial_x^\alpha \tilde{V}_k(x) - \partial_x^\alpha \tilde{V}_k(y) \big| = f_k^{-2} l_k^{5}  \big|(\partial_x^\alpha V)(l_kx+x_k) - (\partial_x^\alpha V)(l_k y+x_k) \big|  \leq C_V |x-y|.
	\end{aligned}
\end{equation*}
For $\alpha\in\N^d_0$ with $|\alpha|\geq1$ we have that
\begin{equation*}
	\begin{aligned}
    	 \norm{\partial_x^\alpha \tilde{a}_{l,k}(x) }_{L^\infty(B(0,8))} =  l_k^{|\alpha|-1} \sup_{x\in B(0,8)}  |(\partial_x^\alpha a_j)(l_kx+x_k)|  \leq   \norm{\partial_x^\alpha a_j(x) }_{L^\infty(\R^d)},
	\end{aligned}
\end{equation*}
for all $j\in\{1,\dots,d\}$. Both bounds are independent of $k$ and $\hbar$. 
The last numbers we check are the numbers $\norm{\partial_x^\alpha \widetilde{\varphi_k\varphi}}_{L^\infty(\R^d)}$ for all $\alpha\in\N_0^d$. Here we have by construction of $\varphi_k$ (\eqref{EQ:Global_five_derivative_2.5}) for all $\alpha\in\N_0^d$ that
\begin{equation*}
	\begin{aligned}
   	\norm{\partial_x^\alpha \widetilde{\varphi_k\varphi}}_{L^\infty(\R^d)}
   	&=\sup_{x\in\R^d} \abs{l_k^{\abs{\alpha}} \sum_{\beta\leq \alpha} {\binom{\alpha}{\beta}} (\partial_x^{\beta}\varphi_k)(l_k x+x_k) (\partial_x^{\alpha - \beta}\varphi)(l_kx+x_k) }
    	\\
    	&\leq C_\alpha \sup_{x\in\R^d} \sum_{\beta\leq \alpha} {\binom{\alpha}{\beta}} l_k^{\abs{\alpha-\beta} }\abs{(\partial_x^{\alpha - \beta}  \varphi)(l_kx+x_k) } \leq \widetilde{C}_\alpha.
      \end{aligned}
  \end{equation*}
With this we have established that all numbers the constant from Theorem~\ref{THM:Loc_mod_prob} depend on are independent of $\hbar$ and $k$. From applying Theorem~\ref{THM:Loc_mod_prob} we get that
\begin{equation}\label{EQ:Global_five_derivative_6}
	\begin{aligned}
	\MoveEqLeft \big| \Tr[\varphi  g_\gamma (\mathcal{H}_{\hbar,\mu})  ] - \frac{1}{(2\pi\hbar)^d} \int_{\R^{2d}}g_\gamma( p^2+V(x))\varphi(x) \,dx dp \big|
	\\
	\leq {}& \sum_{k\in\mathcal{I}}\big|  \Tr[\varphi_k\varphi  g_\gamma (\mathcal{H}_{\hbar,\mu})  ] - \frac{1}{(2\pi\hbar)^d} \int_{\R^{2d}}g_\gamma( p^2+V(x))\varphi_k\varphi(x) \,dx dp \big|
	\\
	\leq  {} & \sum_{k\in\mathcal{I}} f_k^{2\gamma}  \big|\Tr[ g_\gamma(\mathcal{H}_{\hbar_k,\mu_k}) \widetilde{\varphi_k\varphi} ]
	 - \frac{1}{(2\pi h_k)^d} \int_{\R^{2d}} g_\gamma( p^2+\tilde{V}_k(x))\widetilde{\varphi_k\varphi}(x) \,dx dp \big|
	\\
	\leq {} & C \sum_{k\in\mathcal{I}} h_k^{1+\gamma-d}f_k^{2\gamma}.
	\end{aligned}
\end{equation}
When we consider the sum over the error terms we see that
\begin{equation}\label{EQ:Global_five_derivative_7}
	\begin{aligned}
	 \sum_{k\in\mathcal{I}} C h_k^{1+\gamma-d}f_k^{2\gamma} &=  \sum_{k\in\mathcal{I}} \tilde{C} \hbar^{1+\gamma-d} \int_{B(x_k,l_k)} l_k^{-d} f_k^{2\gamma}(l_kf_k)^{d-1-\gamma} \,dx
	\\
	& =  \sum_{k\in\mathcal{I}} \tilde{C} \hbar^{1+\gamma-d} \int_{B(x_k,l_k)} l_k^{\gamma-d} l_k^{\frac{3d-3-3\gamma}{2}} \,dx 
	\\
	&\leq \sum_{k\in\mathcal{I}} \hat{C} \hbar^{1+\gamma-d} \int_{B(x_k,l_k)} l(x)^{\frac{d -3 -\gamma}{2}}\,dx
	\leq C   \hbar^{1+\gamma-d} ,
	\end{aligned}
\end{equation}
where we have used the definition of $f_k$ and that $l$ is slowly varying. Moreover, we have also used our assumption on the dimension in the above estimate. In the last inequality, we have used that $\supp(\varphi)$ is compact. This ensures that the constant is finite. Combining the estimates in \eqref{EQ:Global_five_derivative_6} and \eqref{EQ:Global_five_derivative_7} we obtain the desired estimate. This concludes the proof.
\end{proof}
We are now ready to give a proof of our main Theorem. Most of the work has already been done in establishing Theorem~\ref{THM:Local_five_derivative_almost_there}. When comparing to Theorem~\ref{THM:Local_five_derivative_almost_there} what remains in establishing Theorem~\ref{THM:Local_five_derivative} is to allow $\mu\leq C\hbar^{-1}$ for some positive constant and not be bounded by $1$. The argument is identical to the argument used in \cite{MR1343781} to allow $\mu\leq C\hbar^{-1}$ for some positive constant and not be bounded by $1$. We have included it for the sake of completeness. 
\begin{proof}[Proof of Theorem~\ref{THM:Local_five_derivative}]
Since the Theorem has already been established for $\mu\leq\mu_0<1$ we can without loss of generality assume that $\mu\geq \mu_0$, where $\mu_0<1$. We will use the same scaling technique as we used in the proof of Theorem~\ref{THM:Local_five_derivative_almost_there}. Again we have a $\epsilon>0$ such that
\begin{equation*}
	\dist(\supp(\varphi), \Omega^{c}) >\epsilon
\end{equation*} 
since $\varphi\in C_0^\infty(\Omega)$. This time, however, we let 
\begin{equation*}
	l(x) = \min\big(1,\tfrac{\epsilon}{11}\big) \frac{ \mu_0}{ \mu} \quad\text{and}\quad f(x)=1. 
\end{equation*} 
We can again use Lemma~\ref{LE:partition_lemma} with $l$ from above to construct the partition of unity for $\supp(\varphi)$. After this, we do the rescaling as above with unitary conjugations. In this case, we get 
\begin{equation*}
	h_k = \frac{\hbar}{l_kf_k} \leq \hbar \mu \leq C  \quad\text{and}\quad \mu_k = \frac{\mu l_k}{f_k} =  \mu \min\big(1,\tfrac{\epsilon}{11}\big) \frac{ \mu_0}{ \mu} \leq \mu_0.
\end{equation*}
Moreover, we can analogously to above verify that all norm bounds are independent of $k$, $\mu$, and $\hbar$. So after rescaling, we have operators satisfying the assumptions of Theorem~\ref{THM:Local_five_derivative_almost_there} from applying this theorem we get analogous to the calculation in \eqref{EQ:Global_five_derivative_6} that
\begin{equation}\label{EQ:main_proof_1}
	\begin{aligned}
	 \big| \Tr[\varphi  g_\gamma (\mathcal{H}_{\hbar,\mu})  ] - \frac{1}{(2\pi\hbar)^d} \int_{\R^{2d}}g_\gamma( p^2+V(x))\varphi(x) \,dx dp \big|
	\leq  C \sum_{k\in\mathcal{I}} h_k^{1+\gamma-d}.
	\end{aligned}
\end{equation}
Since $\mathcal{I}$ is a finite set we have by our choice of the functions $l$ and $f$ that
\begin{equation}\label{EQ:main_proof_2}
	\begin{aligned}
	 \ \sum_{k\in\mathcal{I}} h_k^{1+\gamma-d} \leq C \mu^{1+\gamma} \hbar^{1+\gamma-d},
	\end{aligned}
\end{equation}
where $C$ depends on $\mu_0$, $\epsilon$ and the number of elements in $\mathcal{I}$. Combining the estimates in \eqref{EQ:main_proof_1} and \eqref{EQ:main_proof_2} we obtain that
\begin{equation}
	\begin{aligned}
	 \big| \Tr[\varphi  g_\gamma (\mathcal{H}_{\hbar,\mu})  ] - \frac{1}{(2\pi\hbar)^d} \int_{\R^{2d}}g_\gamma( p^2+V(x))\varphi(x) \,dx dp \big|
	\leq  C  \mu^{1+\gamma} \hbar^{1+\gamma-d}.
	\end{aligned}
\end{equation}
Recalling the results from Theorem~\ref{THM:Local_five_derivative_almost_there} we get for all $\mu\leq C\hbar^{-1}$ that
\begin{equation}
	\begin{aligned}
	 \big| \Tr[\varphi  g_\gamma (\mathcal{H}_{\hbar,\mu})  ] - \frac{1}{(2\pi\hbar)^d} \int_{\R^{2d}}g_\gamma( p^2+V(x))\varphi(x) \,dx dp \big|
	\leq  C \langle \mu \rangle^{1+\gamma} \hbar^{1+\gamma-d},
	\end{aligned}
\end{equation}
where $\langle \mu \rangle = (1+|\mu|^2)^{\frac{1}{2}}$. This concludes the proof.
\end{proof}

\begin{remark}\label{RE:proof_non_sharp_thm}
The first difference in proving Theorem~\ref{THM:Local_five_derivative} and Theorem~\ref{THM:Main_non_sharp} arises in the preliminary result Theorem~\ref{THM:Local_five_derivative_almost_there}. In the proof of Theorem~\ref{THM:Local_five_derivative_almost_there} the only place we have used our assumptions on the dimension is in equation \eqref{EQ:Global_five_derivative_7}. Here we used the assumption on the dimension to ensure that the number $\frac{d -3 -\gamma}{2}\geq0$. Now if $d=2$ or $d=3$ we have that $\frac{d -3 -\gamma}{2}\leq0$. Hence we can not just use that $l(x)$ is a bounded function, but we have to use that $l(x)\geq\hbar^{\frac{2}{3}}$. So instead of  \eqref{EQ:Global_five_derivative_7} we get the estimate
\begin{equation}
	\begin{aligned}
	 \sum_{k\in\mathcal{I}} C h_k^{1+\gamma-d}f_k^{2\gamma} &\leq \sum_{k\in\mathcal{I}} \hat{C} \hbar^{1+\gamma-d} \int_{B(x_k,l_k)} l(x)^{\frac{d -3 -\gamma}{2}}\,dx
	\leq C   \hbar^{\frac{2}{3}(\gamma-d)},
	\end{aligned}
\end{equation}
with the notation used in the proof of Theorem~\ref{THM:Local_five_derivative_almost_there}. After having established this inequality one obtains  a version of Theorem~\ref{THM:Local_five_derivative_almost_there}  with $d=2$ or $d=3$ with the estimate
\begin{equation*}
	  \Big|\Tr[\varphi g_\gamma(\mathcal{H}_{\hbar,\mu})] - \frac{1}{(2\pi\hbar)^d} \int_{\R^{2d}}g_\gamma(p^2+V(x)) \varphi(x) \,dx dp \Big| \leq C \hbar^{\frac{2}{3}(\gamma-d)}.
\end{equation*}   
The dependents of the constant are the same as stated in the theorem. What remains to establish Theorem~\ref{THM:Main_non_sharp} is to follow the proof of Theorem~\ref{THM:Local_five_derivative} with \eqref{EQ:main_proof_1} replaced by
\begin{equation}\label{EQ:non_sharp_proof_1}
	\begin{aligned}
	 \big| \Tr[\varphi  g_\gamma (\mathcal{H}_{\hbar,\mu})  ] - \frac{1}{(2\pi\hbar)^d} \int_{\R^{2d}}g_\gamma( p^2+V(x))\varphi(x) \,dx dp \big|
	\leq  C \sum_{k\in\mathcal{I}} h_k^{\frac{2}{3}(\gamma-d)},
	\end{aligned}
\end{equation}
where one have used the version of Theorem~\ref{THM:Local_five_derivative_almost_there} described above. Since we instead of \eqref{EQ:main_proof_1} have the estimate given in \eqref{EQ:non_sharp_proof_1} we obtain instead of \eqref{EQ:main_proof_2} the estimate
\begin{equation}
	\begin{aligned}
	 \ \sum_{k\in\mathcal{I}} h_k^{\frac{2}{3}(\gamma-d)} \leq C \mu^{\frac{d+2\gamma}{3}} \hbar^{\frac{2}{3}(\gamma-d)}.
	\end{aligned}
\end{equation}
With this estimate, the proof of Theorem~\ref{THM:Main_non_sharp} follows that of Theorem~\ref{THM:Local_five_derivative}.
\end{remark}

\bibliographystyle{plain} \bibliography{Bib_paperB.bib}
\end{document}